\documentclass[11pt]{amsart}
\usepackage[T1]{fontenc}
\usepackage{ dsfont}
\usepackage{amsmath,amsthm, amssymb, latexsym,}
\usepackage{cases}
\usepackage{mathrsfs}
\usepackage{hyperref}
\usepackage{pifont}
\usepackage{amsfonts}
\usepackage{stmaryrd}
\usepackage{color}
\usepackage[all]{xy}
\usepackage[english]{babel}
\usepackage[latin1]{inputenc}

\usepackage[all]{xy}
\newtheorem{theorem}{Theorem}[section]
\newtheorem{lemma}[theorem]{Lemma}
\newtheorem{corollary}[theorem]{Corollary}
\newtheorem{proposition}[theorem]{Proposition}

\theoremstyle{definition}

\theoremstyle{remark}
\newtheorem{remark}[theorem]{Remark}

\numberwithin{equation}{section}

\newcommand{\on}[1]{\operatorname{#1}}
\newcommand{\abs}[1]{\lvert#1\rvert}

\newcommand{\un}[1]{\underline{#1}}

\newcommand{\mc}[1]{\mathcal{#1}} 
\newcommand{\mf}[1]{\mathfrak{#1}} 
\newcommand{\xg}{\backslash} 
\newcommand{\bd}{\bullet} 

\newcommand{\lieg}{\mathfrak{g}} 

\newcommand{\ds}{\oplus} 
\newcommand{\ten}{\otimes} 

\newcommand{\fz}{\mathbb{Z}} 
\newcommand{\fn}{\mathbb{N}} 
\newcommand{\ff}{\mathbb{F}} 

\newcommand{\ra}{\rightarrow} 

\newcommand{\iso}{\simeq} 

\newcommand{\pr}[1]{#1^{\prime}}
\newcommand{\SL}{\on{SL}}
\newcommand{\GL}{\on{GL}}
\newcommand{\n}[1]{N_{\tau({#1}),#1}}
\newcommand{\nmi}[1]{N_{m_{i-1},#1}}
\newcommand{\nmii}[1]{N_{m_{i},#1}}
\begin{document}
\title[Modular Invariants]{Modular Invariants for some Finite Modular Pseudo-Reflection Groups}
\author{Ke Ou} \address[Ke Ou]{School of Statistics and Mathematics, Yunnan University of Finance and Economics, Kunming 650221, China} \email{keou@ynufe.edu.cn}
\date{\today}
\begin{abstract}
	In this paper, we determine the modular invariants for some pseudo-reflection subgroups of the finite general linear group $ \GL_n(q) $ acting on the tensor product of the symmetric algebra $ S^{\bd}(V) $ and the exterior algebra $ \wedge^{\bd}(V) $ of the natural $ \GL_n(q) $-module $ V $. 
	We are particularly interested in the case when the pseudo-reflection groups are subgroups of the parabolic subgroup of $ \GL_n(q) $ which are generalization of Weyl groups of Cartan type Lie algebra.	
\end{abstract}
\maketitle
\section{Introduction}
Let $ p $ be a fixed prime and $ \ff_q $ be the finite field with $ q = p^r $ for some $ r\geq 1. $ The finite
general linear group $ \GL_n(q) $ acts naturally on the symmetric algebra $ \mc{P}:=S^{\bd}(V) $, where $ V=\ff_q^n $ is the standard $ \GL_n(q) $-module.
Let $ G $ be a subgroup of $ \GL_n(q). $ Then $ G $ also acts as a group of graded algebra automorphisms of $ \mc{P} $, and the collection of $ G $ invariant polynomials forms a graded algebra $ \mc{P}^G $, the algebra of invariants.
Suppose $ M $ is a finite dimensional $ \ff_qG $-module. Then one gets an induced $ G $-action on the free $ \mc{P} $-module $ \mc{P}\ten M $. The set $ (\mc{P}\ten M)^G, $ consisting of all
$ G $-invariant elements in $ \mc{P}\ten M,  $ forms a $ \mc{P}^G $-module which is called the module of covariants of type $ M $ by Broer and Chuai (\cite{BC}).
In particular, when $ M=\wedge^{\bd}(V), $ the exterior algebra
on $ V, $ one will be interested in the graded algebraic structure on $ \mc{A}:=S^{\bd}(V)\ten \wedge^{\bd}(V) $ and the
module structure of $ \mc{A}^G $ over $ \mc{P}^G $, stemming from the computation of mod $ p $ cohomology rings of finite $ p $-groups and the study of invariant differential forms in geometry (see \cite[Chapter 5]{Ben}  for example). For the last half century, modular invariant theory has been playing an important role in algebraic topology. We refer to \cite{Sm} and \cite{Wi} for expositions on Dickson invariants and connections to topology. 

The $ \GL_n(q) $ invariants in $ \mc{P} $ (resp. $ \mc{A} $) are determined by Dickson \cite{Di} (resp.  Mui \cite{Mu}).
For a composition $ I=(n_1,\cdots,n_l) $ of $ n, $ let $ \GL_I(q) $ be the parabolic subgroup of $ \GL_n(q) $ associated to $ I $ which of the form
\begin{equation}
	\GL_I(q)=\left( 
	\begin{matrix}\GL_{n_1}(q) & * & \cdots & *\\ 
		0 & \GL_{n_2}(q) &\cdots & *\\
		\vdots & \vdots & \vdots & \vdots\\
		0 & 0 & \cdots & \GL_{n_l}(q)
	\end{matrix}
	\right).
\end{equation}
Generalizing \cite{Di}, Kuhn and Mitchell \cite{KM} showed that the algebra $ \mc{P}^{\GL_I(q)} $ is a polynomial algebra in $ n $ explicitly given generators. Minh and T\`{u}ng \cite{MT} determined the $ \GL_I(q) $ invariants of $ \mc{A} $ in the case $ q = p $, as they  used some Steenrod algebra arguments. Wan and Wang \cite{WW} generalized to relative $ \GL_I(q) $ invariants of $ \mc{A} $ for an arbitrary $ q. $

For given subgroups $ G_i $ of $ \GL_{n_i}(q),\ 1\leq i\leq l, $ let $ G_I $ and $ U_I $ be the subgroup of $ \GL_I(q) $ which is of the form
\begin{equation}\label{G}
G_I=\left( 
\begin{matrix}G_1 & * & \cdots & *\\ 
	0 & G_2 &\cdots & *\\
	\vdots & \vdots & \ddots & \vdots\\
	0 & 0 & \cdots & G_l
\end{matrix}
\right)
\text{ and }
U_I=\left( 
\begin{matrix}I_{n_1} & * & \cdots & *\\ 
	0 & I_{n_2} &\cdots & *\\
	\vdots & \vdots & \ddots & \vdots\\
	0 & 0 & \cdots & I_{n_l}
\end{matrix}
\right) \text{ respectively,}
\end{equation}
where $ I_j $ is the identity matrix of $ \GL_j(q). $ These groups can be viewed as a special class of polynomial glueing groups that originally introduced by Huang in \cite{H} and recently have been reformulated in \cite{CSW}. In particular, the  $ G_I $ invariants of $ \mc{P} $ can be determined by the Gluing Lemma of \cite{H} or \cite[Theorem 3.10]{CSW}.

In this paper, we study the $ G_I $ invariants of $ \mc{A} $ when $ G_i $ is $ I_{n_i}, $ general linear group $ \GL_{n_i}(q),$ special linear group $ \SL_{n_i}(q)$ and imprimitive pseudo-reflection group $ G(m,a,n_i) $ (see Section \ref{6.2.1} for the definition). 
Since $ U_I $ is a normal subgroup of $ G_I, $ we have $ \mc{P}^{G_I}\iso (\mc{P}^{U_I})^{L_I} $ and $ \mc{A}^{G_I}\iso (\mc{A}^{U_I})^{L_I} $ where $ L_I=G_I/U_I $ is the quotient group of $ U_I $ in $ G_I. $ As a result, the $ U_I $ invariants of $ \mc{P} $ and $ \mc{A} $ play the central roles. For more details see equation (\ref{U_I on A}) and Theorem \ref{main thm}. Our investigation generalizes the results on modular invariants of $ \mc{A} $ by  Minh-T\'{u}ng \cite{MT} and Mui \cite{Mu}. 

Note that $ G_I $ is a  generalization of  $ \GL_I(q) $ as well as the Weyl groups of Cartan type Lie algebras. Namely, $ G_I=\GL_I(q) $ if $ G_i= \GL_{n_i}(q) $ for all $ i. $ And $ G_I $ becomes a Weyl group of Cartan type Lie algebras if $ l=2,\ q=p,$  $ G_1=\GL_{n_1}(q),\ G_2=S_{n_2} $ or $ S_{n_2}\ltimes \fz_2^{n_2} $ (cf. \cite{Je} or \cite[Proposition 2.2]{CO}, see Section \ref{6.111} for the definition). If $ G $ is a Weyl group of Cartan type Lie algebra, the invariants $ \mc{P}^G $ and $ \mc{A}^G $ which are determined in Corollary \ref{Weyl gp inv} will help us to understand the semisimple orbits and Chevalley restriction theorem for Cartan type Lie algebras (cf. \cite{CO}). 
Meanwhile, $ G_I $ is a modular finite pseudo-reflection group  if $ l\geq 2 $ and all $ G_i $ are pseudo-reflection groups since $ p\mid |U_I| $ (see Section \ref{PRG}). 

Let $ G $ be a finite subgroup of $ \GL_n(q) $ and $ M $ be a finite dimensional module. In the non-modular case, a lot more is known than in the modular case. For example, the following results are true in the non-modular situation but not so in the modular situation. The algebra $ \mc{P}^G $ is a Cohen-Macaulay and $ (\mc{P}\ten M)^G $ a Cohen-Macaulay module over $ \mc{P}^G $ (\cite{HE}). The invariant algebra $ \mc{P}^G $  is a polynomial algebra if and only if $ G $ is a pseudo-reflection group (this goes back to Chevalley, Shephard, Todd and Bourbaki, see  \cite[Theorem 7.2.1]{Ben}, \cite[Theorem 18-1]{Ka}), and all modules of covariants are free when $ G $ is a pseudo-reflection group.
This is one of the main reasons why modular rings of invariants are more difficult to handle.

Our first main result is the following statement.

\begin{theorem}
	Let $ I=(n_1,\cdots,n_l) $ be a composition of $ n $. Then $ \mc{A}^{U_I} $ is a free module of rank $ 2^n $ over the algebra $ \mc{P}^{U_I}. $
\end{theorem}
We refer to Theorem \ref{main thm} for a more precise version of Theorem 1.1 where an explicit basis for the free module is given. Theorem 1.1 in the special case when $ l=n $ and $ I=(1,\cdots,1) $ 
is due to Mui \cite{Mu}, and our approach toward Theorem 1.1 (or rather Theorem \ref{main thm}) is a generalization of \cite{Mu}. We will then discuss $ (\mc{A}^{U_I})^{G_i} $ in Section 6 case by case where $ G_i=\GL_{n_i}(q),\ \SL_{n_i}(q) $ or $ G(m,a,n_i) $. Then we have the following statement.
\begin{theorem}
	Let $ I=(n_1,\cdots,n_l) $ be a composition of $ n $. Suppose $ p>n_i $ whence $ G_i=G(r_i, a_i, n_i). $
	\begin{enumerate}
		\item If $ G_i=G(r_i,a_i,n_i) $ such that $ r_i\mid q-1 $ for all $ i=1,\cdots,l. $ Then $ \mc{A}^{G_I} $ is a free module of rank $ 2^na_1\cdots a_l $ over the algebra $ \mc{P}^{\overline{G_I}} $, where $ \overline{G_I}=\left(G(r_1,1,n_1)\times\cdots\times G(r_l,1,n_l) \right)\ltimes U_I. $
		\item If there is $ 0\leq a\leq l $ such that 		
		\[ G_i=\left\{
		\begin{array}{ll}
		\GL_{n_i}(q) & i=1,\cdots,a \\
		G(r_i,1,n_i) & i=a+1,\cdots ,l
		\end{array}
		\right., \]
		then $ \mc{A}^{G_I} $ is a free module of rank $ 2^n $ over the algebra $ \mc{P}^{G_I}. $
	\end{enumerate}
\end{theorem} 
For more details and explicit basis of these free modules, we refer to Theorem \ref{imprimi 2} for part 1, Theorem \ref{imprimitive} for the case $ a=0 $ of  part 2 and Theorem \ref{6.3} for the case $ 1\leq a\leq l $ of part 2. 

For module of covariants of type $ M, $ Broer and Chuai provide the so called Jacobian criterion of freeness for $ (\mc{P}\ten M)^G $ as $ \mc{P}^G $ module, see \cite[Theorem 3]{BC}. Our approach does not rely on this result, since it seems too hard to verify the Jacobian criterion in our case.

The paper is organized as follows. Some notations and conventions are provided in Section 2.1. The necessary concepts and results concerning pseudo-reflection groups are recalled in Section 2.2. In Section 2.3, we first recall the Dickson invariants and Kuhn-Mitchell invariants of $ \mc{P}, $ then the invariants of $ \mc{P} $ over $ G_I $ are determined by \cite{CSW}. In Section 3, we recall the results by Mui, Ming-T\`{u}ng and Wan-Wang invariants of $ \mc{A}. $ Most concepts and results in these sections can be found in \cite{CSW, Di,Ka,KM,MT,Mu,WW}. The $ U_I $ and $ G_I $ invariants of $ \mc{A} $ are given in Sections 4, 5 and 6. Precisely, Section 4 deals with $ \mc{A}^{U_I} $ and Sections 5, 6 describe $ \mc{A}^{G_I} $ for the concrete $ G_i. $

\section{Preliminary}
\subsection{}
Throughout this paper, let $ I=(n_1,\cdots,n_l) $ be a fixed composition of $ n $. 

Set $ m_0=0$ and $ m_k=\sum_{i=1}^{k}n_i,\ k=1,\cdots,l. $ 
For each $ 1\leq s\leq n, $ define $$ \tau(s)=m_j \text{ if } m_j<s\leq m_{j+1}. $$

Recall the definition of $ G_I $ and $ U_I $ given by formula (\ref{G}). One can check that $ G_I=L_I\ltimes U_I, $  where \[
L_I=\left( 
\begin{matrix}G_1 & 0 & \cdots & 0\\ 
	0 & G_2 &\cdots & 0\\
	\vdots & \vdots & \ddots & \vdots\\
	0 & 0 & \cdots & G_l
\end{matrix}
\right).
\]

Moreover, since $ U_I $ is a normal subgroup of $ G_I, $ we have $ \mc{P}^{G_I}\iso (\mc{P}^{U_I})^{L_I} $ and $ \mc{A}^{G_I}\iso (\mc{A}^{U_I})^{L_I}. $

Suppose $ V=\ff_q^n, $ the symmetric algebra  $ S^\bd (V) $ and the exterior algebra $ \wedge^\bd (V) $ will be identified with $ \ff_q[x_1,\cdots, x_n] $ and $ E[y_1,\cdots, y_n] $ respectively. Namely, $ \mc{P}=\ff_q[x_1,\cdots, x_n] $ and $ \mathcal{{A}}= \ff_q[x_1,\cdots, x_n] \ten E[y_1,\cdots, y_n]. $ Then $ \mc{A} $ is an associative superalgebra with a $ \fz_2 $-gradation induced by the trivial $ \fz_2 $-gradation of $ \ff_q[x_1,\cdots, x_n] $ and the natural $ \fz_2 $-gradation of $ E[y_1,\cdots, y_n]. $ 
Denote by $ \text{d}(f) $ the parity of $ f\in \mc{A}. $ 

Set $ \mathds{B}(n)=\bigcup_{k=0}^{n}\mathds{B}_k $ where $ \mathds{B}_0=\emptyset $ and $ \mathds{B}_k=\{ (i_1,\cdots,i_k)\mid 1\leq i_1<\cdots<i_k\leq n \}. $ Then $ E[y_1,\cdots, y_n] $ has a basis $ \{ y_J\mid J\in \mathds{B}(n) \} $ where $ y_J=y_{j_1}\cdots y_{j_t} $ if $ J=(j_1,\cdots,j_t)\neq \emptyset $ and $ y_{\emptyset}=1. $ 

For $ 0\leq k\leq n $ and $ I,J\in \mathds{B}_k, $ we say that $ I<J $ if 
there is $ 1\leq l\leq k $ such that $ i_l<j_l $ and $ i_{s}=j_s, $ for all $ l<s\leq k. $
Moreover, $ I\leq J $ if $ I=J $ or $ I<J. $

One can check that $ (\mathds{B}_k,\leq) $ is a total order on $ \mathds{B}_k $ for all $ 1\leq k\leq n. $

For $ K=(k_1,\cdots,k_t)\in\mathds{B}_t, $ and $ a,a_i\in \{ 1,\cdots,n \}\xg \{k_1,\cdots,k_t\}, $ define 
\begin{itemize}
	\item $ K+\{ a \} := (\cdots,k_s,a,k_{s+1},\cdots) $ if $ k_s<a<k_{s+1}; $	
	\item $ K+\{ a_1,\cdots,a_s \}:= (\cdots((K+\{ a_1 \})+ \{a_2\}) \cdots); $
	\item $ K-\{k_j\}:=(k_1,\cdots,\widehat{k_j},\cdots,k_t); $
	\item $ K-\{k_{j_1},\cdots, k_{j_s}\}:= (\cdots((K-\{ k_{j_1} \})- \{k_{j_2}\}) \cdots); $
	\item $ \tau(K):= \left\{
	\begin{array}{ll}
	\tau(k_t) & K\notin \mathds{B}_0 \\
	0 & K\in \mathds{B}_0
	\end{array}\right.;
$
	\item $\on{hd}{(K)}:=K-\{ k_j\mid k_j\leq \tau(K) \}. $ Namely, $ \on{hd}{(K)}=(k_{i+1},\cdots,k_t) $ if $ k_i\leq \tau(K)<k_{i+1}. $
\end{itemize}

\subsection{Pseudo-reflection groups}\label{PRG}
In this subsection, we will recall some basic facts for pseudo-reflection groups. More details refer to \cite{Ka}.

For a finite dimensional vector space $ W $ over $ \ff_q $, a pseudo-reflection is a linear isomorphism $ s: W \ra W $ that is not the
identity map, but leaves a hyperplane $  H\subseteq W $ pointwise invariant. $ G \subseteq \GL(W) $ is a pseudo-reflection group if $ G $ is generated by its
pseudo-reflections. We call $ G $ is non-modular if $ p\nmid \abs{G} $ while $ G $ is modular otherwise.
\begin{lemma}
	$ G_I $ is a  pseudo-reflection group if all $ G_i $ are pseudo-reflection groups.
\end{lemma}
\begin{proof}
	Let $ J $ (resp. $ K $) be the set consisting of all pseudo-reflections of $ G_1\times\cdots \times G_l $ (resp. all elementary matrices of $ U_I $). One can check that $ G_I $ can be generated by $ J\cup K. $
\end{proof}
\begin{remark}
	As a corollary, the Weyl groups of Cartan type Lie algebras with type $ W,\ S $ and $ H $ are modular pseudo-reflection groups.
\end{remark}


\begin{lemma}\label{6.4}\cite[Corollary 3.1.4]{CW}
	Suppose both $ H_1 $ and $ H $ are non-modular pseudo-reflection groups and
	$ H_1 $ is a subgroup of $ H. $ Then 
	$ S^{\bd}(W)^{H_1} $ is a free $ S^{\bd}(W)^H $ module of rank $ \frac{|H|}{|H_1|}. $ 
\end{lemma}

\subsection{Invariants of $ \mc{P} $}
For $ 1\leq k\leq n $ and $ 0\leq i\leq k, $ define homogeneous polynomials $ V_k,\  L_k$ and $ L_{k,i} $ in variables $ x_1,\cdots, x_k $:
\[ V_k= \prod_{\lambda_1,\cdots, \lambda_{{k-1}} \in\ff_q} (\lambda_1x_1+\cdots\lambda_{k-1}x_{k-1}+x_k), \]
\[ L_k=\prod_{i=1}^k V_i = \prod_{i=1}^k \prod_{\lambda_1,\cdots, \lambda_{{i-1}} \in\ff_q} (\lambda_1x_1+\cdots+\lambda_{i-1}x_{i-1}+x_i), \]
\[L_{k,i}=\left| 
\begin{matrix}x_1 & x_2 & \cdots & x_k\\
	x_1^q & x_2^q & \cdots & x_k^q\\
	\vdots & \vdots & \vdots & \vdots\\ 
	\widehat{x_1^{q^i}} & \widehat{x_2^{q^i}} & \cdots & \widehat{x_k^{q^i}}\\
	\vdots & \vdots & \vdots & \vdots\\ 
	x_1^{q^{k}} & x_2^{q^{k}} & \cdots & x_k^{q^{k}}
\end{matrix}
\right|,\]
where the hat $\ \widehat{ }\ $ means the omission of the given term as usual.


By \cite{Di},  $ L_k $ is a factor of $ L_{k,i}. $
Define $ Q_{k,i}=L_{k,i}/L_k. $ Then $ \on{deg}(Q_{k,i})=q^k-q^i. $
The subalgebras of invariants over both $ \SL_n(q) $ and  $ \GL_n(q) $ of $ \mc{P} $ are polynomial algebras. Moreover,
\begin{equation}\label{Dickson}
 \mc{P}^{\SL_n(q)} = \ff_q[L_n,Q_{n,1}, \cdots, Q_{n,n-1}], 
\end{equation}
\begin{equation}\label{Dickson2}
\mc{P}^{\GL_n(q)} = \ff_q[Q_{n,0}, \cdots, Q_{n,n-1}].\quad\ \ 
\end{equation}
		
For $ 1\leq i\leq l,\ 1\leq j\leq n_i, $ define 
\begin{equation}\label{vij}
	v_{i,j}=\prod_{\lambda_1,\cdots, \lambda_{m_{i-1}}\in\ff_q} (\lambda_1x_1+\cdots \lambda_{m_{i-1}}x_{m_{i-1}}+x_{m_{i-1}+j}),
\end{equation}
\begin{equation}\label{qij}
	q_{i,j}=Q_{n_i,j}(v_{i,1},\cdots,v_{i,n_i} ).
\end{equation}
Then $ \deg(v_{i,j})=q^{m_{i-1}} $ and $ \deg (q_{i,j})=q^{m_i}-q^{m_i-j}. $ 

Recall the Hilbert series of a graded space $ W^{\bd}=\ds_i W^i $ is by definition the generating function $ H(W^{\bd},t):=\sum_i t^i \dim W^i. $

Similar to the proof of \cite[Lemma 1]{MT},	
\begin{equation}\label{U_I on A}
\mc{P}^{U_I} = \ff_q[x_1,\cdots,x_{n_1}, v_{2,1} ,\cdots, v_{2,n_2},\cdots, v_{l,1},\cdots, v_{l,n_l} ].  
\end{equation}

Moreover, by \cite[Theorem 2.2]{KM} and \cite[Theorem 1.4]{He},
\begin{equation}\label{KM}
\mc{P}^{\GL_I(q)} = \ff_q[q_{i,j}\mid 1\leq i\leq l, 1\leq j\leq n_i],
\end{equation}
\[ H\left(\mc{P}^{\GL_I(q)},t\right)=\frac{1}{\prod_{i=1}^l \prod_{j=1}^{n_i}
(1- t^{q^{m_i}-q^{m_{i}-j}})}. \]

By \cite[Theorem 3.10]{CSW}, the following proposition holds.

\begin{proposition}\label{poly inv}
For $ 1\leq i\leq l, $
assume that $ \ff_q[x_1,\cdots,x_{n_i}]^{G_{i}} =\ff_q[e_{i,1},\cdots, e_{i, n_i}] $ is a polynomial algebra such that $ \deg(e_{i,j})=\alpha_{ij}. $ For $ 1\leq j\leq n_i, $ define
	$  u_{i,j}=e_{i,j}(v_{i,1},\cdots,v_{i,n_i}).  $ The subalgebra $ \mc{P}^{G_I} $ of $ G_I $-invariants in $ \mc{P} $ is a polynomial ring on the generators $ u_{i,j} $ of degree $ \alpha_{ij}\cdot q^{m_{i-1}} $ with  $ 1\leq i\leq l, 1\leq j\leq n_i. $ Namely,
	$$  \mc{P}^{G_I} = \ff_q[u_{i,j}\mid 1\leq i\leq l, 1\leq j\leq n_i]. $$
	Moreover, the Hilbert series of $ \mc{P}^G $ is 
	\[ H\left(\mc{P}^{G_I},t\right)=\frac{1}{\prod_{i=1}^l \prod_{j=1}^{n_i}
		(1- t^{\alpha_{ij}\cdot q^{m_{i-1}}})}. \]
\end{proposition}

\section{Mui, Ming-T\`{u}ng and Wan-Wang Invariants of $ \mc{A} $}
In this section, we will recall the work of Mui, Ming-T\`{u}ng and Wan-Wang invariants in $ \mc{A}. $ 

\subsection{Mui invariants of $ \mc{A} $}
By \cite[equation (1.4)]{Mu}, let $ A=(a_{ij}) $ be a $ n\times n $ matrix with entries in a possibly noncommutative ring $ R $. Define the (row) determinant of $ A $:
\[ |A|=\det(A)=\sum_{\sigma\in S_n} \text{sgn}(\sigma) a_{1\sigma(1)} \cdots a_{n\sigma(n)}. \]

Recall that $ ab=(-1)^{\on{d}(a)\on{d}(b)}ba $ for all $ a,b \in \mc{A}. $ One must revise the usual computation of the determinants in this case.
For example, one can compute easily from the definition that, in the algebra
$ E_{\fz}(y_1,\cdots,y_n)=\ds_{J\in \mathds{B}(n)}\fz y_J $ over $ \fz $, 
\[  \frac{1}{n!}\left| 
\begin{matrix}y_1 & y_2 & \cdots & y_n\\
y_1 & y_2 & \cdots & y_n\\
\vdots & \vdots & \ddots & \vdots\\ 
y_1 & y_2 & \cdots & y_n
\end{matrix}
\right|=y_1\cdots y_n ,  \ 
  \left| 
\begin{matrix}y_1 & y_1 & \cdots & y_1\\
y_2 & y_2 & \cdots & y_2\\
\vdots & \vdots & \ddots & \vdots\\ 
y_n & y_n & \cdots & y_n
\end{matrix}
\right|=0.\qquad\quad\quad   \]

Suppose $ 1\leq j\leq m\leq n,$ and let $ (b_1,\cdots,b_j) $ be a sequence of integers such that $ 0\leq b_1<\cdots< b_j\leq m-1. $ Define $ M_{m;b_1,\cdots,b_j} \in \mc{A} $ by the following determinant of $ m\times m $ matrix
\begin{equation}\label{Mui invariant}
M_{m;b_1,\cdots,b_j} =\frac{1}{j!}  \left| 
\begin{matrix}y_1 & y_2 & \cdots & y_m\\
\vdots & \vdots & \vdots & \vdots\\ 
y_1 & y_2 & \cdots & y_m\\
x_1 & x_2 & \cdots & x_m\\
\vdots & \vdots & \vdots & \vdots\\ 
\widehat{x_1^{q^{b_i}}} & \widehat{x_2^{q^{b_i}}} & \cdots & \widehat{x_m^{q^{b_i}}}\\
\vdots & \vdots & \vdots & \vdots\\ 
x_1^{q^{m-1}} & x_2^{q^{m-1}} & \cdots & x_m^{q^{m-1}}
\end{matrix}
\right|
\begin{matrix}
	\big\uparrow\\
	j \text{ rows}\\
	\big\downarrow\\
	\\
	\big\uparrow\\
	\big|\\
	m-j \text{ rows}\\
	\big|\\
	\big\downarrow
\end{matrix}.
\end{equation}

Let $ U_n(q) $ be the subgroup of $ \GL_n(q) $ consisting of all upper triangular matrices. 

By \cite[Theorem 4.8, Theorem 4.17, Theorem 5.6]{Mu},
	\[ \mc{A}^{\SL_n(q)}= \ff_q[L_n,Q_{n,1}, \cdots, Q_{n,n-1}] \ds \sum_{j=1}^{n}\sum_{0\leq b_1<\cdots<b_j\leq n-1} M_{n;b_1,\cdots,b_j} \ff_q[L_n,Q_{n,1}, \cdots, Q_{n,n-1}],   \]
	\[ \mc{A}^{\GL_n(q)}= \ff_q[Q_{n,0}, \cdots, Q_{n,n-1}] \ds \sum_{j=1}^{n}\sum_{0\leq b_1<\cdots<b_j\leq n-1} M_{n;b_1,\cdots,b_j}L_n^{q-2} \ff_q[Q_{n,0}, \cdots, Q_{n,n-1}],   \]
\begin{equation}\label{A^ U_n}
\mc{A}^{\on{U}_n(q)}= \ff_q[V_1, \cdots, V_n] \ds \sum_{k=1}^{n} \sum_{s=k}^{n} \sum_{0\leq b_1<\cdots<b_k= s-1} M_{s;b_1,\cdots,b_k} \ff_q[V_1, \cdots, V_n].  
\end{equation}

\subsection{$ \GL_I(q) $-Invariants of Minh-T\`{u}ng and Wan-Wang of $ \mc{A} $}
For $ 1\leq i\leq l, $ define $ \theta _i $ by letting
\[ \theta_i=L_{n_i}(v_{i,1},v_{i,2},\cdots, v_{i,n_i}). \]

The following result in the case $ q=p $ is \cite[Theorem 3]{MT} and in general $ q $ is \cite[Theorem 3.1]{WW}.
\begin{theorem}
	$ \mc{A}^{\GL_I(q)} $ is a free $ \mc{P}^{\GL_I(q)} $ module of rank $ 2^n, $ with a basis consisting of $ 1 $ and $ M_{m_i;b_1,\cdots,b_j}\theta_1^{q-2}\cdots \theta_i^{q-2} $ for $ 1\leq i\leq l,1\leq j\leq m_i $ and $ 0\leq b_1<\cdots < b_j\leq m_i-1,$  $ b_j\geq m_{i-1}. $ Namely,
	\[ \mc{A}^{\GL_I(q)}= \mc{P}^{\GL_I(q)} \ds \sum_{j=1}^{n}\sum_{m_i\geq j}\sum_{\substack{
			0\leq b_1<\cdots<b_j\leq m_i-1\\m_{i-1}\leq b_j}
	} M_{m_i;b_1,\cdots,b_j}\theta_1^{q-2}\cdots \theta_i^{q-2}\mc{P}^{\GL_I(q)}.   \]
\end{theorem}

\section{$ U_I $-invariants of $ \mc{A} $}
In this section, we will investigate the $ U_I $ (cf. notation (\ref{G})) invariants of $ \mc{A} $ which generalize Mui's invariants (cf. formula (\ref{A^ U_n})).

Let $ 1\leq b\leq n $ and $ S=( s_1,\cdots,s_k,a_1,\cdots,a_t )\in\mathds{B}_{k+t} $ such that $ s_k\leq b<a_1. $ 

If $ S\neq \emptyset, $ define 
\[ N_{b,S} :=\frac{1}{(k+t)!}  \left| 
\begin{matrix}y_1 & \cdots & y_b & y_{a_1} & y_{a_2} & \cdots & y_{a_t}\\
\vdots & \vdots & \vdots & \vdots& \vdots & \vdots & \vdots\\ 
y_1 & \cdots & y_b & y_{a_1} & y_{a_2} & \cdots & y_{a_t}\\
x_1 &  \cdots & x_b & x_{a_1} & x_{a_2} & \cdots & x_{a_t}\\
\vdots & \vdots & \vdots & \vdots& \vdots & \vdots & \vdots\\ 
\widehat{x_1^{q^{s_1-1}}}  & \cdots & \widehat{x_b^{q^{s_1-1}}} & \widehat{x_{a_1}^{q^{s_1-1}}} & \widehat{x_{a_2}^{q^{s_1-1}}} & \cdots & \widehat {x_{a_t}^{q^{s_1-1}}}\\
\vdots & \vdots & \vdots & \vdots& \vdots & \vdots & \vdots\\ 
\widehat{x_1^{q^{s_k-1}}}  & \cdots & \widehat{x_n^{q^{s_k-1}}} & \widehat{x_{a_1}^{q^{s_k-1}}} & \widehat{x_{a_2}^{q^{s_k-1}}} & \cdots & \widehat {x_{a_t}^{q^{s_k-1}}}\\
\vdots & \vdots & \vdots & \vdots& \vdots & \vdots & \vdots\\ 
x_1^{q^{b-1}} &  \cdots & x_b^{q^{b-1}} & x_{a_1}^{q^{b-1}} & x_{a_2}^{q^{b-1}} & \cdots & x_{a_t}^{q^{b-1}}
\end{matrix}
\right|
\begin{matrix}
\big\uparrow\\
k+t \text{ rows}\\
\big\downarrow\\

\big\uparrow\\
\big|\\
\big|\\
\big|\\
b-k \text{ rows }\\
\big|\\
\big|\\
\big|\\
\big\downarrow
\end{matrix}.
 \]
For convenience, define $ N_{b,\emptyset}:=1. $
Sometimes, we denote $ N_{b,S} $ as $N^{\un{s}}_{b,\un{a}}, $ where $ \un{s}=(s_1,\cdots,s_k)\in \mathds{B}_k $ and $ \un{a}=(a_1,\cdots,a_t)\in\mathds{B}_t. $  Then $ N_{b,S}\in S^{u}(V)\ten \wedge^{k+t}(V), $ where $ u=\frac{q^b-1}{q-1}-\sum_{i=1}^k q^{s_i-1}. $

\begin{remark}
	Suppose $ 1\leq j\leq m\leq n,$ and $ B=(b_1+1,\cdots, b_j+1)\in \mathds{B} _j $ such that $ 0\leq b_1<\cdots< b_j\leq m-1. $ Then $ N_{m,B}=M_{m;b_1,\cdots,b_j} $ (cf. formula (\ref{Mui invariant})). 
\end{remark}

For $ J=(1,\cdots,b,a_1,\cdots a_t)\in \mathds{B}_ {b+t}, $ one can check by definition that 
\begin{equation}\label{pure odd}
N_{b,J}=y_J=y_1\cdots y_by_{a_1}\cdots y_{a_t}.
\end{equation}

For $ 1\leq b<a\leq n, $ denote 
$$ V_{b,a}=L_{b+1}(x_1,\cdots,x_b,x_a)/L_b(x_1,\cdots, x_b) = \prod_{\lambda_1,\cdots, \lambda_{b}\in\ff_q} (\lambda_1x_1+\cdots \lambda_{b}x_{b}+x_{a}). $$
Then $ v_{i,j}=V_{m_{i-1}, m_{i-1}+j} $ by formula (\ref{vij}).

\begin{lemma}\label{4.1}
	Suppose $ \un{s}=( s_1,\cdots,s_k)\in\mathds{B}_{k}$ and $\un{a}=(a_1,\cdots,a_t )\in\mathds{B}_{t} $ such that $ s_k\leq b<a_1. $
	
$ \on{(1)} $ If $ b+1<a_1, $ then
$$ N^{\un{s}}_{b,\un{a}}\cdot V_{b+1}=(-1)^t N^{\un{s}}_{b+1,\un{a}}+ \sum_{i=1}^{t}(-1) ^{i+1} N^{\un{s}+\{b+1\}}_{b+1,\un{a}- \{a_i\}} V_{b,a_i}$$
$$\quad + \sum_{j=1}^{k}(-1) ^{k+t+j} N^{\un{s}+\{b+1\}-\{s_j\}}_{b+1,\un{a}} Q_{b,s_j}. $$

$ \on{(2)} $ If $ b+1=a_1, $ i.e. $ b=a_1-1, $ then we have
$$ N^{\un{s}}_{b,\un{a}}\cdot V_{a_1+1}=(-1)^{t-1} N^{\un{s}+\{a_1\}}_{a_1+1,\un{a}- \{a_1\}} + \sum_{i=2}^{t}(-1) ^{i} N^{\un{s}+\{ a_1,a_1+1 \}}_{a_1+1,\un{a}- \{ a_1,a_i \}} V_{a_1,a_i}$$
$$+ \sum_{j=1}^{k}(-1) ^{k+t+j} N^{\un{s}+\{ a_1,a_1+1\}- \{s_j\}}_{a_1+1, \un{a}-\{ a_1 \}}Q_{a_1,s_j}.\quad $$
\end{lemma}
\begin{proof}
(1) We consider the following determinant :
\[D_1= \frac{1}{(k+t)!}  \left| 
\begin{matrix}
x_1 &  \cdots & x_b &  x_1 & \cdots  & x_{b+1} & x_{a_1} & x_{a_2} & \cdots & x_{a_t}\\
\vdots & \vdots & \vdots & \vdots  & \vdots & \vdots & \vdots& \vdots & \vdots & \vdots\\ 
x_1^{q^b} &  \cdots & x_b^{q^b} & x_1^{q^b} &  \cdots  & x_{b+1}^{q^b} & x_{a_1}^{q^b} & x_{a_2}^{q^b} & \cdots & x_{a_t}^{q^b}\\
0&\cdots&0 & y_1 & \cdots  & y_{b+1} & y_{a_1} & y_{a_2} & \cdots & y_{a_t}\\
\vdots&\vdots&\vdots& \vdots & \vdots & \vdots & \vdots& \vdots & \vdots & \vdots\\
\vdots&\vdots&\vdots&  y_1 & \cdots  & y_{b+1} & y_{a_1} & y_{a_2} & \cdots & y_{a_t}\\
\vdots&\vdots&\vdots & x_1 &  \cdots  & x_{b+1} & x_{a_1} & x_{a_2} & \cdots & x_{a_t}\\
\vdots&\vdots&\vdots & \vdots  & \vdots & \vdots & \vdots& \vdots & \vdots & \vdots\\ 
\vdots&\vdots&\vdots & \widehat{x_1^{q^{s_1-1}}}  & \cdots  & \widehat{x_{b+1}^{q^{s_1-1}}} & \widehat{x_{a_1}^{q^{s_1-1}}} & \widehat{x_{a_2}^{q^{s_1-1}}} & \cdots & \widehat {x_{a_t}^{q^{s_1-1}}}\\
\vdots&\vdots&\vdots & \vdots  & \vdots & \vdots & \vdots& \vdots & \vdots & \vdots\\ 
\vdots&\vdots&\vdots & \widehat{x_1^{q^{s_k-1}}}  & \cdots  & \widehat{x_{b+1}^{q^{s_k-1}}} & \widehat{x_{a_1}^{q^{s_k-1}}} & \widehat{x_{a_2}^{q^{s_k-1}}} & \cdots & \widehat {x_{a_t}^{q^{s_k-1}}}\\
\vdots&\vdots&\vdots & \vdots  & \vdots & \vdots & \vdots& \vdots & \vdots & \vdots\\ 
0&\cdots&0 & x_1^{q^{b-1}} &  \cdots  & x_{b+1}^{q^{b-1}} & x_{a_1}^{q^{b-1}} & x_{a_2}^{q^{b-1}} & \cdots & x_{a_t}^{q^{b-1}}
\end{matrix}
\right|  
\begin{matrix}
\big\uparrow\\
b+1 \text{ rows}\\
\big\downarrow\\
\\
\big\uparrow\\
k+t \text{ rows}\\
\big\downarrow\\
\\
\big\uparrow\\
\big|\\
\big|\\
\big|\\
b-k \text{ rows }\\
\big|\\
\big|\\
\big|\\
\big\downarrow
\end{matrix}
\]
By the Laplace expansion, we can expand $ D_1 $ by the first $ b+1 $ rows. Hence,
\[ D_1= (-1)^bL_{b+1}N^{\un{s}}_{b,\un{a}}+  \sum_{i=1}^{t}(-1) ^{b+i} L_{b+1}(x_1,\cdots,x_b,x_{a_i})N^{\un{s}+\{b+1\}}_{b+1,\un{a}- \{a_i\}}. \]
Similarly, by expanding the first $ b $ columns, 
\[D_1 = (-1)^{b+t}L_{b}N^{\un{s}}_{b+1,\un{a}}+ \sum_{i=1}^{k}(-1)^{b+1-s_i}L_{b,s_i}\cdot (-1) ^{k+t+s_i -(i-1)} N^{\un{s}+\{ b+1\}- \{s_j\}}_{b+1, \un{a}}. \]
Combining above equations and divide $ (-1)^bL_b(x_1,\cdots, x_b) $ on both side. Statement (1) holds.

(2) Now we consider the following determinant:
\[ D_2=\frac{1}{(k+t)!}  \left| 
\begin{matrix}
x_1 &  \cdots & x_{a_1} &  x_1 & \cdots & x_{a_1} & x_{a_1+1}  & x_{a_2} & \cdots & x_{a_t}\\
\vdots & \vdots & \vdots & \vdots & \vdots  & \vdots & \vdots& \vdots & \vdots & \vdots\\ 
x_1^{q^{a_1}} &  \cdots & x_{a_1}^{q^{a_1}} & x_1^{q^{a_1}} &  \cdots & x_{a_1}^{q^{a_1}} & x_{a_1+1}^{q^{a_1}}  & x_{a_2}^{q^{a_1}} & \cdots & x_{a_t}^{q^{a_1}}\\
0&\cdots&0& y_1 & \cdots & y_{a_1} & y_{a_1+1}  & y_{a_2} & \cdots & y_{a_t}\\
\vdots&\vdots&\vdots& \vdots & \vdots & \vdots  & \vdots& \vdots & \vdots & \vdots\\
\vdots&\vdots&\vdots&  y_1 & \cdots & y_{a_1} & y_{a_1+1}  & y_{a_2} & \cdots & y_{a_t}\\
\vdots&\vdots&\vdots& x_1 &  \cdots & x_{a_1} & x_{a_1+1}  & x_{a_2} & \cdots & x_{a_t}\\
\vdots&\vdots&\vdots & \vdots & \vdots & \vdots & \vdots& \vdots & \vdots & \vdots\\ 
\vdots&\vdots&\vdots& \widehat{x_1^{q^{s_1}}}  & \cdots & \widehat{x_{a_1}^{q^{s_1}}} & \widehat{x_{a_1+1}^{q^{s_1}}}  & \widehat{x_{a_2}^{q^{s_1}}} & \cdots & \widehat {x_{a_t}^{q^{s_1}}}\\
\vdots&\vdots&\vdots & \vdots & \vdots & \vdots  & \vdots& \vdots & \vdots & \vdots\\ 
\vdots&\vdots&\vdots & \widehat{x_1^{q^{s_k}}}  & \cdots & \widehat{x_{a_1}^{q^{s_k}}} & \widehat{x_{a_1+1}^{q^{s_k}}}  & \widehat{x_{a_2}^{q^{s_k}}} & \cdots & \widehat {x_{a_t}^{q^{s_k}}}\\
\vdots&\vdots&\vdots & \vdots & \vdots & \vdots & \vdots& \vdots & \vdots & \vdots\\ 
0&\cdots&0 & x_1^{q^{a_1}} &  \cdots & x_{a_1}^{q^{a_1-2}} & x_{a_1+1}^{q^{a_1-2}} & x_{a_2}^{q^{a_1-2}} & \cdots & x_{a_t}^{q^{a_1-2}}
\end{matrix}
\right|  
\begin{matrix}
\big\uparrow\\
a_1+1 \text{ rows}\\
\big\downarrow\\
\\
\big\uparrow\\
k+t \text{ rows}\\
\big\downarrow\\
\\
\big\uparrow\\
\big|\\
\big|\\
\big|\\
a_1-1-k \text{ rows }\\
\big|\\
\big|\\
\big|\\
\big\downarrow
\end{matrix}.
\]
Similar to the proof of statement (1), by expanding the first $ b+1 $ rows,
\[ D_2=(-1)^{a_1}L_{a_1+1} N^{\un{s}}_{b,\un{a}}+  \sum_{i=2}^{t}(-1) ^{a_1+i+1} L_{a_1+1}(x_1,\cdots,x_{a_1},x_{a_i})N^{\un{s}+\{a_1,a_1+1\}}_{a_1+1,\un{a}- \{a_1,a_i\}}. \]
By expanding the first $ b $ columns,
\[ D_2= (-1)^{a_1-1+t}L_{a_1}N^{\un{s}+\{a_1\}}_{a_1+1,\un{a}- \{a_1\}}+ \sum_{j=1}^{k}(-1)^{a_1+1+k+t-(j-1)}L_{a_1,s_j}N^{\un{s}+\{a_1,a_1+1\}- \{s_j\}}_{a_1+1, \un{a}-\{a_1\}}. \]
Combining them and divide $ (-1)^{a_1}L_{a_1} $ on both side. Statement (2) holds.
\end{proof}
By direct computation, one have the following result.
\begin{corollary}
	For $ J=( j_1,\cdots,j_t)\in \mathds{B}_t $ and $ 1\leq b< j_t, $ we have
	\begin{enumerate}
		\item $ N_{\tau (J),J} $ is $ U_I $-invariant.
		\item If $ b\neq j_s-1, $ for all $ s=1,\cdots,t, $ then $$ N_{b,J}\cdot V_{b+1} =\epsilon N_{b+1,J}+\sum_{i=1}^{t} g_iN_{b+1, J+ \{b+1\}- \{j_i\}}  $$
		where $ \epsilon\in\{ \pm1 \} $ and $ g_i\in \mc{P}^{U_I}. $
		
\noindent		If $ b=j_s-1, $ for some $ s=1,\cdots,t, $ then $$ N_{b,J}\cdot V_{b+1} =\epsilon N_{b+1,J}+\sum_{i=1}^{t} g_iN_{b+1, J+ \{j_s+1\}- \{j_i\}}  $$
		where $ \epsilon\in\{ \pm1 \} $ and $ g_i\in \mc{P}^{U_I}. $
	\end{enumerate}
\end{corollary}
\begin{remark}
	For arbitrary $ b $ and $ J $, $ N_{b,J} $ may not be $ U_I $-invariant.

\end{remark}

\begin{corollary}\label{action on V's}
	Let $ 1\leq b\leq  c\leq n $ and $ J= (j_1,\cdots,j_t)\in\mathds{B}_{t}. $ 
	
	If $ j_i\leq b<j_{i+1}\leq j_l\leq  c< j_{l+1}, $ 
	then
	\[ N_{b,J}\cdot V_{b+1}\cdots\widehat{V_{j_{i+1}}}\cdots \widehat{V_{j_l}}\cdots V_{c}= \epsilon N_{c, J}+ \sum_{\pr{J}} N_{c,\pr{J}}f_{\pr{J}} \]
	where $ \epsilon\in\{\pm 1\}, $  $ \pr{J}\leq  (1,\cdots,j_i,c-l+i+1,\cdots,c,j_{l+1} , \cdots, j_t) $ and $ f_{\pr{J}}\in \mc{P}^{U_I}. $
\end{corollary}
\begin{proof}
	For any $ K\in\mathds{B}(n) $ and $ d\in K,$ it is a direct computation that  $$ N_{d-1,K}= N_{d,K}. $$ Thanks to Lemma \ref{4.1}, one can check this corollary by induction.
\end{proof}

\begin{remark}
	Note that $ J<(1,\cdots,j_i,c-l+i+1,\cdots,c,j_{l+1} , \cdots, j_t). $
\end{remark}
	
We may denote $ N_{b,s}= N_{b,S} $ if $ S=(s)\in\mathds{B}_1. $

\begin{lemma}\label{linear indepen 1}
	If $ S=(s_1,\cdots, s_k)\in \mathds{B}_k, $ and $ s_j\leq b < s_{j+1}, $ then 
	\[ N_{b,S}=(-1)^{jk-(j+1)j/2}N_{b,s_1}\cdots N_{b,s_k}/L_b^{k-1}. \]
	In particular, if $ s_i\leq \tau(S)< s_{i+1}, $ then 
	\[ N_{\tau(S),S}=(-1)^{ik-(i+1)i/2}N_{\tau(S),s_1}\cdots N_{\tau(S),s_k}/L_{\tau(S)}^{k-1}. \]	
\end{lemma}
\begin{proof}
	The relation holds trivially for $ k=1. $ Let us suppose $ k>1 $ and that it is true for all $ N_{a,J} $ where $ 1\leq a\leq n $ and $ J\in \mathds{B}_{k-1}. $
	
	Now we consider the following determinant:
	\[ D=\left| 
	\begin{matrix}
	y_1 &  \cdots & y_b &  y_1 & \cdots & y_b & y_{s_{j+1}} & \cdots & y_{s_k}\\
	x_1 &  \cdots & x_b &  x_1 & \cdots & x_b & x_{s_{j+1}} & \cdots & x_{s_k}\\
	\vdots & \vdots & \vdots & \vdots & \vdots & \vdots& \vdots & \vdots & \vdots\\ 
	x_1^{q^{b-1}} &  \cdots & x_b^{q^{b-1}} & x_1^{q^{b-1}} &  \cdots & x_b^{q^{b-1}} & x_{s_{j+1}}^{q^{b-1}} & \cdots & x_{s_k}^{q^{b-1}}\\
0&\cdots & 0 & y_1 & \cdots & y_b & y_{s_{j+1}} & \cdots & y_{s_k}\\
\vdots& \vdots & \vdots & \vdots & \vdots & \vdots & \vdots & \vdots & \vdots\\
\vdots& \vdots & \vdots &  y_1 & \cdots & y_b & y_{s_{j+1}} & \cdots & y_{s_k}\\
\vdots& \vdots & \vdots & x_1 &  \cdots & x_b & x_{s_{j+1}} & \cdots & x_{s_k}\\
\vdots& \vdots & \vdots & \vdots & \vdots & \vdots & \vdots & \vdots & \vdots\\ 
\vdots& \vdots & \vdots & \widehat{x_1^{q^{s_1-1}}}  & \cdots & \widehat{x_b^{q^{s_1-1}}}   & \widehat{x_{s_{j+1}}^{q^{s_1-1}}} & \cdots & \widehat {x_{s_k}^{q^{s_1-1}}}\\
\vdots& \vdots & \vdots & \vdots & \vdots & \vdots & \vdots & \vdots & \vdots\\ 
\vdots& \vdots & \vdots & \widehat{x_1^{q^{s_j-1}}}  & \cdots & \widehat{x_b^{q^{s_j-1}}}  & \widehat{x_{s_{j+1}}^{q^{s_j-1}}} & \cdots & \widehat {x_{s_k}^{q^{s_j-1}}}\\
\vdots& \vdots & \vdots & \vdots & \vdots & \vdots  & \vdots & \vdots & \vdots\\ 
0&\cdots & 0 & x_1^{q^{b-1}} &  \cdots & x_b^{q^{b-1}} & x_{s_{j+1}}^{q^{b-1}} & \cdots & x_{s_k}^{q^{b-1}}
	\end{matrix}
	\right|  .
	\begin{matrix}
	\Big\uparrow\\
	b+1 \text{ rows}\\
	\Big\downarrow\\
	\\
	\big\uparrow\\
	k-1 \text{ rows}\\
	\big\downarrow\\
	\\
	\big\uparrow\\
	\big|\\
	\big|\\
	\big|\\
	b-j \text{ rows }\\
	\big|\\
	\big|\\
	\big|\\
	\big\downarrow
	\end{matrix}
	\]

By expanding $ D $ by the first $ b $ columns, we have
\[ D= (-1)^bk!L_bN_{b,S}+ \sum_{i=1}^j (-1)^{b+k-i+1}N_{b,s_i} (k-1)!N_{b,S-\{ s_i \}}.\] 

By expand $ D $ by the first $ b+1 $ rows, we have
\[D=\sum_{i=j+1}^k (-1)^{b+1+i-j}N_{b,s_i} (k-1)!N_{b,S-\{ s_i \}}. \]

Therefore, we obtain:
\[ kL_bN_{b,S}=\sum_{i=1}^{j}(-1)^{k-i}N_{b,s_i}N_{b,S-\{ s_i \}} + \sum_{i=j+1}^k (-1)^{i-j+1}N_{b,s_i}N_{b,S-\{ s_i \}}. \]

From the induction hypothesis, we have

\bigskip
\begin{tabular}{rcl}
	$ kL_b^{k-1}N_{b,S} $& $ = $ & $ \sum_{i=1}^{j}(-1)^{k-i}N_{b,s_i}\cdot (-1)^{(k-1)(j-1)-j(j-1)/2} N_{b,s_1}\cdots \widehat{ N_{b,s_i}}\cdots N_{b,s_k} $\\
	&&\\
	& $ + $ & $ \sum_{i=j+1}^k (-1)^{i-j+1}N_{b,s_i}\cdot (-1)^{(k-1)j-j(j+1)/2} N_{b,s_1}\cdots \widehat{ N_{b,s_i}}\cdots N_{b,s_k} $\\
	&&\\
	& $ = $ & $ (-1)^{jk-(j+1)j/2}kN_{b,s_1}\cdots N_{b,s_k}. $
\end{tabular}

\bigskip
Consequently, $ L_b^{k-1}N_{b,S}=(-1)^{jk-(j+1)j/2}N_{b,s_1}\cdots N_{b,s_k}. $
Lemma holds.
\end{proof}

\begin{corollary}\label{linear indepen 2}
	If $ S=(s_1,\cdots,s_k)\in \mathds{B}_k $ and $ b<s_1, $ then 
	\[ N_{b,1}\cdots N_{b,b}N_{b,s_1}\cdots N_{b,s_k}=(-1)^{bk-b(b+1)/2} L_b^{b+k-1} y_1\cdots y_by_{s_1}\cdots y_{s_k}. \]
\end{corollary}
\begin{proof}
	Thanks to above lemma and equation (\ref{pure odd}), we have
	\[ (-1)^{bk-b(b+1)/2}N_{b,1}\cdots N_{b,b}N_{b,S}=L_b^{b+k-1} N_{b, J}= L_b^{b+k-1} y_1\cdots y_by_{s_1}\cdots y_{s_k}, \]
	where $ J=(1,\cdots,b,s_1,\cdots,s_k)\in \mathds{B}_{b+k}. $
\end{proof}

\begin{corollary}\label{linear indepen 3}
	For all $ 1\leq b,s\leq n, N_{b,s}^2=0. $
\end{corollary}
\begin{proof}
	If $ b\geq s, $ then   $ N_{b,1}\cdots \widehat{ N_{b,s}}\cdots N_{b,b}N_{b,s}^2=\pm L_b^{b-1} y_1\cdots y_bN_{b,s}=0.  $
	
	If $ b< s, $ then   $ N_{b,1}\cdots N_{b,b}N_{b,s}^2=\pm L_b^{b} y_1\cdots y_by_{s}N_{b,s}=0.  $
	
	By Lemma \ref{linear indepen 1}, $ N_{b,1}\cdots N_{b,b}=(-1)^{jk-(j+1)j/2} L_b^{b-1} y_1\cdots y_b \neq 0. $ Corollary holds.
\end{proof}

Similar arguments with \cite[Lemma 5.2]{Mu}, by Corollary \ref{action on V's}, Lemma \ref{linear indepen 1}, Corollary \ref{linear indepen 2} and \ref{linear indepen 3}, the following proposition holds.
\begin{proposition}\label{base}
	Let $  f=\sum_{J\in\mathds{B}(n)} N_{\tau(J),J}h_J $ 
	where $ h_J\in \mc{P}. $ Then $ f=0 $ if and only if all $ h_J=0. $
\end{proposition}

\begin{lemma}\label{inv of head}
Suppose $ J_*=(j_1,\cdots,j_k)\in\mathds{B}_k, $ and $ f=\sum_{J\leq J_*} y_Jf_J(x_1,\cdots x_n)\in\mc{A} $ is $ U_I $-invariant, then $ f_{J_*}\in \mc{P} $ is $ U_I $-invariant. Moreover, $ f_{J_*} $ has factors
\[ \big\{V_i\mid i\in \{1,\cdots,\tau(j_k)\}\xg \{ j_1,\cdots, j_k \} \big\}. \]
\end{lemma}
\begin{proof}
	For all $ w=(w_{ij})\in U_I,\ wy_i=y_i+ w_{i-1,i}y_{i-1}+\cdots +w_{1i}y_1. $ Therefore,
	\[ wf=\sum_{J<J_*} y_J\pr{f}_J+y_{J_*}\pr{f}_{J_*}, \]
	where $ wf_{J_*}=\pr{f_{J_*}}. $ Comparing the coefficient of $ y_{J_*} $ of $ wf=f, $ we have $ wf_{J_*}=f_{J_*}. $
	
	Now, for each $ i\in{J_*}\cap \{1,\cdots,\tau(j_k)\},\ E+E_{i,j_k}\in U_I. $ Hence, $ (E+E_{i,j_k})\cdot f =f.$ Denote $ K=J_*+\{i\}-\{j_k\}. $ Comparing the coefficient of $ y_{K} $ on both side, we have
	\[ y_{J_*-\{j_k\}}y_if_{J_*}(x_1,\cdots, x_i+x_{j_k},\cdots) + y_Kf_K(x_1,\cdots,x_i+x_{j_k},\cdots)= y_Kf_K, \]
	where $ x_i+x_{j_k} $ is the $ j_k $-th component. Then
	\[ \epsilon f_{J_*}(x_1,\cdots, x_i+x_{j_k},\cdots)=f_K -  f_K(x_1,\cdots,x_i+x_{j_k},\cdots), \]
	where $ \epsilon\in\{ \pm1 \}. $ Taking value $ x_i=0 $ on both side, then
	$  f_{J_*}(\cdots,x_{i-1},0,x_{i+1},\cdots)=0.  $ Therefore, $ f_{J_*} $ has factor $ x_i. $ Since $ f_{J_*} $ is $ U_I $-invariant and all $ E+E_{j,i}\in U_I,\ 1\leq j< i, $ $ f_{J_*} $ has factor $ V_i. $
\end{proof}
	
\begin{proposition}\label{main prop}
Suppose $ S_*=(s_1^*,\cdots,s_k^*)\in\mathds{B}_k $ and $ s_j^*\leq\tau(s_k^*)< s_{j+1}^* $ for some $ 1\leq j\leq k-1. $ Let $ f=\sum_{S\leq S_*} y_Sf_S(x_1,\cdots x_n)\in\mc{A} $ be $ U_I $-invariant. Then
\[ f=\sum_{L\leq \on{hd}{(S_*)}} \sum_{\substack{S=(s_1,\cdots,s_k)\\ (s_{j+1},\cdots,s_k)=L}} N_{\tau(s_k),S}h_S(x_1,\cdots,x_n), \] 
where $ h_S\in \mc{P} $ is $ U_I $-invariant.
\end{proposition}
\begin{proof}
	Suppose $ s_k^*=b. $ We will use double induction on both $ k $ and $ S_*. $
	
	\bigskip
(1) Suppose $ k=1 $ and $ S_*=(b), 1\leq b\leq n. $ 

\bigskip
(i)If $ b=1, $  $ \tau(b)=0. $ Moreover, $  N_{\tau(1), 1}=y_1 $ and $ f=y_1f_1. $ By Lemma \ref{inv of head}, $ f_1\in \mc{P}^{U_I} $ and proposition holds.

\bigskip
(ii) For arbitrary $ b, $ denote $ c=\tau(b). $ Suppose $ f=y_1f_1+\cdots y_bf_b. $ By Lemma \ref{inv of head}, $ f_b $ is $ U_I $-invariant and has factors $ \{ V_i\mid 1\leq i\leq c \}. $ Therefore, $ f_b=(-1)^{c+1}y_bL_{c}h_b $ where $ h_b\in \mc{P}^{U_I}. $ The expension of $ N_{c,b} $ along row 1 implies that
\[ N_{c,b}= (-1)^{c+1}y_bL_c+\sum_{i=1}^{c}(-1)^{i+1}y_iN_{i} \] where $ N_i\in \mc{P} $ is the minor of $ N_{c,b} $ at position $ (1,i). $ Hence $  f=N_{c,b}h_b +\sum_{i=1}^{b-1}y_i\pr{f}_i.  $ Note that $ f-N_{c,b}h_b =\sum_{i=1}^{b-1}y_i\pr{f}_i $ is $ U_I $-invariant. By induction, there are $ h_i\in \mc{P}^{U_I} $ such that
\[ f-N_{c,b}h_b=\sum_{i=1}^{b-1}N_{\tau(i),i}h_i\text{ and } f=\sum_{i=1}^{b}N_{\tau(i),i}h_i. \]

\bigskip
(2) For arbitrary $ k>1, $ suppose $ s_{k-1}^*=l<b, $ and $ s_i^*\leq\tau(l)< s_{i+1}^*. $

\bigskip
(i) If $ b=k, $ i.e. $ S_*=(1,2,\cdots,k), $ then $ f=y_{S_*}f_{S_*}. $ Note that $ y_{S_*}=N_{\tau(k),S_*} $ is $ U_I $-invariant. For all $ w\in U_I, $
$ wf=y_{S_*}(w\cdot f_{S_*})=y_{S_*}f_{S_*}, $ and hence $ f_{S_*} $ is $ U_I $-invariant. 

Proposition holds in this case.

\bigskip
(ii) Let us suppose $ b>k $ and that it is true for all $ S<S_*. $ One can rewrite $ f $ as
 \begin{equation}\label{f}
f=\left(\sum_{K\leq K_*}y_KF_K\right)y_b+\sum_{\substack{b\not\in S\\ S\leq S_*}}y_S f_S,
\end{equation} 
where $ K_*=(s_1^*,\cdots,s_{k-1}^*)\in\mathds{B}_{k-1} $ and $ F_K= f_{K+ \{b\}}. $

Now, set $ F=\sum_{K\leq K_*}y_KF_K. $  
Define $$ T(K_*)=\{(\alpha_1,\cdots,\alpha_{i},s_{i+1}^*,\cdots,s_{k-1}^*)\}\subseteq \mathds{B}_{k-1}. $$

Similar to the proof of Lemma \ref{inv of head}, one can prove that $ F $ is $ U_I $-invariant. Then by induction, $ F $ can be decomposed into 
\begin{equation}\label{decom F}
F=\sum_{L\leq \on{hd}{(K_*)}} \sum_{\substack{K=(s_1,\cdots,s_{k-1})\\ (s_{i+1},\cdots,s_{k-1})=L}} N_{\tau(s_{k-1}),K}h_K(x_1,\cdots,x_n)
\end{equation} 
where all $ h_K $ are $ U_I $-invariant.

Note that $ y_{S_*}f_{S_*}=y_{K_*}y_bF_{K_*}. $ As a component of $ F, $ $ N_{\tau(s),K} $ has factor $ y_{K_*} $ if and only if $ K\in T(K_*) $ which equivalent to $ L=\on{hd}{(K_*)}. $

Thanks to Lemma \ref{inv of head}, $ f_{S_*} $ has factors $ V_{\tau(l)+1}\cdots\widehat{V_{s_{i+1}^*}}\cdots \widehat{V_{s_{j}^*}} \cdots V_{\tau(b)}. $ It is a direct computation that $ N_{\tau(l),K} $ has no such factors if $ K\in T(K_*). $ As a consequence, $$ h_K= V_{\tau(l)+1}\cdots \widehat{V_{s_{i+1}^*}} \cdots \widehat{V_{s_{j}^*}} \cdots V_{\tau(b)}  \pr{h}_K$$ where $ \pr{h}_K\in \mc{P} $ for all $ K\in T(K_*). $ Since all of $ h_K $ and $ V_i\ (\tau(l)+1\leq i\leq \tau(b)) $ are $ U_I $-invariant, $ \pr{h}_K $ is also $ U_I $-invariant.

Denote $ \tilde{K_*}=(\tau(b)-j,\cdots, \tau(b)-1,s_{j+1}^*,\cdots, s_{k-1}^*). $ Thanks to Corollary \ref{action on V's},  
\[ \sum_{K\in T(K_*)} N_{\tau(l),K}V_{\tau(l)+1}\cdots \widehat{V_{s_{i+1}^*}} \cdots \widehat{V_{s_{j}^*}} \cdots V_{\tau(b)}= \sum_{S\leq \tilde{K_*}} N_{\tau(b),S}f_{S} \]
where $ f_{S}\in \mc{P}^{U_I}. $

Then
\begin{equation}\label{F as V component}
F=\sum_{L<\on{hd}{(K_*)}} \sum_{\substack{K=(s_1,\cdots,s_{k-1})\\ (s_{i+1},\cdots,s_{k-1})=L}} N_{\tau(s_{k-1}),K}h_K + \sum_{S\leq \tilde{K_*} } N_{\tau(b),S}h_{S}
\end{equation}
where $ h_S\in \mc{P}^{U_I} $ since all $ f_S\ (S\leq \tilde{K_*}) $ and $ \pr{h}_K\ (K\in T(K_*)) $ are $ U_I $-invariant.

For each $ S=(s_1,\cdots,s_j,s_{j+1}^*,\cdots, s_{k-1}^*)\leq \tilde{K_*}, $ note that $$ \on{hd}(S+\{b \}) = \on{hd}(S_*) = (s_{j+1}^*,\cdots, s_{k-1}^*, b). $$
By Laplace expansion,
\[ N_{\tau(b),S}y_b=(-1)^{u\cdot \tau(b)}y_{\on{hd}(S_*)}N_{\tau(b), (s_1,\cdots,s_j)} + \sum_{\pr{S}< S_*}y_{\pr{S}}\alpha_{\pr{S}}, \]
\[ N_{\tau(b),S+\{b \}}=(-1)^{(u+1)\cdot \tau(b)}y_{\on{hd}(S_*)} N_{\tau(b), (s_1, \cdots, s_j)} + \sum_{\pr{S}< S_*}y_{\pr{S}}\beta_{\pr{S}} \]
where $ u=k-j-1, $ $ \alpha_{\pr{S}},\beta_ {\pr{S}}\in \mc{P}. $ Therefore, 
\begin{equation}\label{N_b}
N_{\tau(b),S}y_b = (-1)^{\tau(b)} N_{\tau(b),S+\{b \}} + \sum_{\pr{S}< S_*} y_{\pr{S}}\gamma_{\pr{S}}
\end{equation}  where $ \gamma_{\pr{S}} = \alpha_{\pr{S}}- (-1)^{\tau (b)} \beta_{\pr{S}} \in \mc{P}. $

Combining equation (\ref{f}), (\ref{F as V component}) and (\ref{N_b}), we have
\begin{equation}\label{4.5}
f=\sum_{\substack{ S\leq \tilde{K_*}\\ \on{hd}(S+\{b\})=\on{hd}(S_*)}}  N_{\tau(b),S+\{b \}}h_S+A+B+C+D
\end{equation}
where $ h_S\in \mc{P}^{U_I} $ and
\[  A = \sum_{\substack{b\not\in S\\ S\leq S_*}} f_{S,1}y_S, \]
\[ B=\sum_{L<\on{hd}{(K_*)}} \sum_{\substack{S=(s_1,\cdots,s_{k-1})\\ (s_{i+1},\cdots,s_{k-1})=L}} N_{\tau(s_{k-1}),S}f_{S,2}y_b, \]
\[ C=\sum_{\substack{ S\leq \tilde{K_*}\\ \on{hd}(S+\{b\})<\on{hd}(S_*)}} N_{\tau(b),S+\{b \}}f_{S,3}, \]
\[ D=\sum_{\pr{S}< S_*} y_{\pr{S}}\gamma_{\pr{S}} \]
such that all possible $ f_{S,i} \in \mc{P},\ i=1,2,3, $ and $ \gamma_{\pr{S}}\in \mc{P}. $ 

It is obviously that $ A+B+C+D= \sum_{S<S_*} y_S\pr{f}_S $ where $ \pr{f}_S \in \mc{P} $ for all possible $ S. $ If $ S\leq \tilde{K_*}$ and $ \on{hd}(S+\{b\})=\on{hd}(S_*), $ then $ S+\{b\}=(s_1,\cdots,s_k) $ such that $ (s_{j+1},\cdots,s_k)= (s_{j+1}^*,\cdots,s_k^*). $ Therefore, one can rewrite equation (\ref{4.5}) as:
\[ f=\sum_{\substack{S=(s_1,\cdots,s_k)\\ (s_{j+1},\cdots,s_k)=\on{hd}(S_*)}} N_{\tau(b),S}h_S+ \sum_{K<S_*} y_K \pr{f}_K\]
where $ h_S\in \mc{P}^{U_I} $ and $ \pr{f}_K\in \mc{P}. $ Since both $ f $ and $ \sum_{\substack{S=(s_1,\cdots,s_k)\\ (s_{j+1},\cdots,s_k)=\on{hd}(S_*)}} N_{\tau(b),S}h_S $ are $ U_I $-invariant, then $ \sum_{K<S_*} y_K \pr{f}_K $ is also $ U_I $-invariant. Hence, proposition holds by induction.
\end{proof}

By equation (\ref{U_I on A}), Proposition \ref{base} and Proposition \ref{main prop}, we have the following main theorem.
\begin{theorem}\label{main thm}
	\begin{enumerate}
		\item $ \mc{P}^{U_I}= \ff_q[x_1,\cdots,x_{n_1}, v_{2,1} ,\cdots, v_{2,n_2},\cdots, v_{l,1},\cdots, v_{l,n_l} ], $
		\item    $ \mc{A}^{U_I} $ is a free  $ \mc{P}^{U_I}$ module of rank $ 2^n $ with a basis consisting of all elements of $$ \{ N_{\tau(S),S} \mid S\in\mathds{B}(n) \}. $$ In other words, there exists a decomposition
	\end{enumerate}
	\[ \mc{A}^{U_I}= \sum_{S\in \mathds{B}(n)} N_{\tau(S),S}\mc{P}^{U_I}.   \]
\end{theorem}
\begin{remark}
	If $ I=(1,\cdots,1), $ i.e. $ U_I=U_n(q) $ the unipotent radical of $ \GL_I(q), $ then $ \tau (j)=j-1,\ j=1,\cdots,n. $ 
	
	Suppose $ 1\leq j\leq m\leq n,$ and $ 0\leq b_1<\cdots< b_j= m-1. $ Then $ M_{m;b_1,\cdots b_j}=N_{m-1,B}=N_{ \tau(B),B} $ where $ B=(b_1+1,\cdots,b_j+1)\in \mathds{B}_{j}. $ Therefore, Theorem \ref{main thm} generalize formula (\ref{A^ U_n}).
\end{remark}

\section{$ G_i $-invariants of $ \mc{A}^{U_I} $}
\subsection{}
For fixed $ 1\leq i\leq l, $ note that $ G_i $ acts on $ x_j $ and $ y_j $ trivially unless $ m_{i-1}<j\leq m_i. $ We will investigate $ (\mc{A}^{U_I})^{G_i} $ in this section.

Suppose $ f(x,y)=\sum_{S\in \mathds{B}(n)} N_{\tau(S),S}f_S(x)\in  (\mc{A}^{U_I})^{G_i} $ where $ x=(x_1,\cdots,x_n), y=(y_1,\cdots,y_n). $ 

Since $ G_i $-action is homogeneous, one can assume that, for some  $ 0\leq k\leq n, $ \begin{equation}\label{5.1}
f=\sum_{S\in \mathds{B}_k} N_{\tau(S),S}f_S\in  (\mc{A}^{U_I})^{G_i}. 
\end{equation}

Moreover, denote $ f=f_1+f_2+f_3, $ where $$ f_1=\sum_{\substack{S\in \mathds{B}_k\\ \tau(S)<m_{i-1}}} N_{\tau(S),S}f_S, $$ $$ f_2=\sum_{\substack{S\in \mathds{B}_k\\ \tau(S)=m_{i-1}}} N_{\tau(S),S}f_S, $$
$$ f_3=\sum_{\substack{S\in \mathds{B}_k\\ \tau(S)\geq m_{i}}} N_{\tau(S),S}f_S. $$
It is a direct computation that $ g\cdot f_i=f_i,\ i=1,2,3, $ for all $ g\in G_i. $

\subsection{} We will describe $ f_1 $ and $ f_3 $ in this subsection.
\begin{lemma}\label{5.11}
	If $ \tau (S)<m_{i-1}, $ then $ f_S $ is $ G_i $ invariant. Moreover, $$ f_1=\sum_{\substack{S\in \mathds{B}_k\\ \tau(S)<m_{i-1}}} N_{\tau(S),S}f_S, \text{ where } f_S\in (\mc{P}^{U_I})^ {G_i}. $$
\end{lemma}
\begin{proof}
	If $ \tau (S)<m_{i-1}, $ then $ \sigma \cdot \n{S}=\n{S} $ for all $ \sigma\in G_i. $ Hence, for every $ g\in G_i $, $$ g\cdot f_1=\sum_{\substack{S\in \mathds{B}_k\\ \tau(S)<m_{i-1}}} (g\cdot N_{\tau(S),S})(g\cdot f_S)=\sum_{\substack{S\in \mathds{B}_k\\ \tau(S)<m_{i-1}}} N_{\tau(S),S}(g\cdot f_S)= \sum_{\substack{S\in \mathds{B}_k\\ \tau(S)<m_{i-1}}} N_{\tau(S),S} f_S. $$ Therefore, $ g\cdot f_S=f_S. $ Lemma holds.
\end{proof}

\begin{lemma}\label{5.2}
	If $ \tau (S)\geq m_{i}, $ then $ f_S $ is $ G_i $ skew-invariant, i.e. $ g\cdot f_S=\det (g)^{-1}f_S $ for all $ g\in G_i. $ Moreover, $$ f_3=\sum_{\substack{S\in \mathds{B}_k\\ \tau(S)\geq m_{i}}} N_{\tau(S),S}f_S, \text{ where } f_S\in \mc{P}^{U_I}\text{ is $ G_i $ skew-invariant}. $$
\end{lemma}
\begin{proof}
	If $ \tau (S)\geq m_{i}, $ one can check that $ g\cdot \n{S}=\det (g)\n{S}. $ 
	$$ g\cdot f_3=\sum_{\substack{S\in \mathds{B}_k\\ \tau(S)\geq m_{i}}} (g\cdot N_{\tau(S),S})(g\cdot f_S)=\sum_{\substack{S\in \mathds{B}_k\\ \tau(S)\geq m_{i}}}\det(g) N_{\tau(S),S}(g\cdot f_S)= \sum_{\substack{S\in \mathds{B}_k\\ \tau(S)\geq m_{i}}} N_{\tau(S),S} f_S. $$
	Therefore, $ g\cdot f_S=\det(g)^{-1}f_S. $ Lemma holds.
\end{proof}

\subsection{}
When $ \tau(S)=m_{i-1}, $ we will discuss case by case. 
\subsubsection{ }\label{6.2.1}$ G_i=G(m,a,n_i)<\GL_{n_i}. $   

Recall that $ G(m,a,n_i)\iso S_{n_i}\ltimes A(m,a,n_i) $ where $ a|m, $
\[ A(m,a,n_i)=\{ \on{diag}(w_1,\cdots,w_{n_i})\mid w_j\in\ff_q,\ w_j^m=(w_1\cdots w_{n_i})^{m/a}=1 \}. \]

Since $  G(m,a,n_i)=G(\pr{m},\pr{a},n_i), $ where $ \pr{m}=(q-1,m),$  $ \pr{a}=\pr{m}/(q-1,m/a), $
one can assume that $ m\mid (q-1) $ and $ m=ab. $ 

Since $ |G_i|=bm^{n_i-1} n_i!, $ therefore $ G_i $ is a nonmodular  group if and only if $ p>n_i $.

For each $ 1\leq i \leq l $ and  $ 1\leq k\leq n_i, $ we need the following notations.
\begin{itemize}
	\item $ \sigma_{i,S}=(m_{i-1}+1,s_1)\cdots (m_{i-1}+k,s_k)\in G(m,a,n_i), $ where $ S:=(s_1,\cdots, s_k)\in\mathds{B}(m_i) $ such that $ s_1>m_{i-1}; $
	\item $ c_{i,k}:=\sum_{ \substack{S=(s_1,\cdots,s_k)\in\mathds{B}_k \\ m_{i-1}<s_1<\cdots<s_k\leq m_i}} \sigma_{i,S}\in \ff_qG(m,a,n_i); $
	\item $ T_{i,k}:=T+ \{ m_{i-1}+1,\cdots, m_{i-1}+k \} \in \mathds{B}(m_{i}) $ for each $ T\in \mathds{B}(m_{i-1}); $
	\item $ \beta_{i,k,r}:=\left\{
	\begin{array}{ll}
	(x_{m_{i-1}+1}\cdots x_{m_{i-1}+k})^{m-1} & \text{ if }r=a\\
	(x_{m_{i-1}+1}\cdots x_{m_{i-1}+k})^{rb-1}(x_{m_{i-1}+k+1}\cdots x_{m_{i}})^{rb} & \text{ if }r=1,\cdots,a-1
	\end{array}
	\right. $ which lies in $ \mc{P}; $
	\item $  H_{i,k}:=G(m,1,k),\ \pr{H}_{i,k}:=G(m,1,n_i-k). $ One can regard $ H_{i,k}\times \pr{H}_{i,k} $ as a subgroup of $ G(m,1,n_i) $ by sending $ (\sigma, \alpha) $ to $ \on{diag}(\sigma, \alpha) $;
	\item By \cite[Section 20-2]{Ka}, if $ p>n_i $, then all skew-invariants of $ \ff_q[x_{m_{i-1}+1}^m\cdots x_{m_{i-1}+k}^m] $ over $ S_k $ form a free $  \ff_q[x_{m_{i-1}+1}^m\cdots x_{m_{i-1}+k}^m]^{S_k} $ module with one generator
	$$ \Delta_{i,k}:= \prod_{m_{i-1}<j_1<j_2\leq m_{i-1}+k} (x_{j_1}^m-x_{j_2}^m). $$ 
\end{itemize}

Recall that  $ \mc{P}^{U_I}=\ten_{i=1}^l \mc{P}_i $
where $ \mc{P}_i=\ff_q[v_{i,1},\cdots,v_{i,n_i}]. $ By Lemma \ref{6.4}, $ (\mc{P}_i)^{ H_{i,k}\times\pr{H}_{i,k}} $ is a free $ \mc{P}_i^{G(m,1,n_i)} $ module of rank $ C_{n_i}^k=\frac{n_i!}{k!(n_i-k)!}. $
Furthermore, suppose $ \{\alpha_{i,k,j}\mid j=1,\cdots,C_{n_i}^k \} $ is a basis.

\begin{lemma}
	\begin{enumerate}
		\item 	$ G_{i,k}:= \on{Stab}_{G_i}(\langle x_{m_{i-1}+1},\cdots,x_{m_{i-1}+k}\rangle) \iso( S_k\times S_{n_i-k})\ltimes A(m,a,n_i). $
		\item For each $ 1\leq k\leq n_i, $ $ G_i $ is generated by $ G_{i,k} $ and all $ \sigma_{i,S} $ where $ S:=(s_1,\cdots, s_k)\in\mathds{B}(n) $ such that $ m_{i-1}<s_1<\cdots<s_k\leq m_i. $
	\end{enumerate}
\end{lemma}
\begin{proof}
	It is a direct computation.
\end{proof}

\begin{lemma}
  $ f_2 $ is $ G_i $ invariant if and only if the following conditions hold for all $ T\in \mathds{B}(m_{i-1}) $ and $ S=(s_1,\cdots,s_k)\in \mathds{B}(m_i) $ such that $ s_1>m_{i-1}. $
	\begin{enumerate}
		\item $ f_{T+ S}(x)=f_{T_{i,k}}(\sigma_{i,S}(x))=\sigma_{i,S}\cdot f_{T_{i,k}}((x)). $ Moreover, $ \nmi{{T+ S}}f_{T+ S} =\sigma_{i,S} (\nmi{T_{i,k}}f_{T_{i,k}}). $
		\item $ \nmi{T_{i,k}}f_{T_{i,k}} $ is $ G_{i,k} $ invariant.
	\end{enumerate}
\end{lemma}

\begin{proof}
	One can check directly that  $ f_2 $ is $ G_i $ invariant if the two conditions hold for all $ T $ and $ S. $
	
	Conversely, suppose $ f_2 $ is $ G_i $ invariant. Then
	
(1) $ \sigma_S\cdot N_{m_{i-1}, T_{i,k}}=N_{m_{i-1}, T+ S} $ and $ \sigma_S(R)=T+ S $ if and onyl if $ R=T_{i,k}; $ 
	
	(2)  $ \sigma\nmii{T_{i,k}}=\chi(\sigma)\nmii{T_{i,k}} $ for some $ \chi(\sigma) \in\ff_q $ and $ \sigma(R)={T_{i,k}} $ if and only if $ R=T_{i,k} $ for each $ \sigma\in G_{i,k}. $ 
	
	Lemma holds. 
\end{proof}

\begin{proposition}\label{S action}
	Assume $ p>n_i. $
	$ {(\mc{A}^{U_I})}^{G(m,a,n_i)} $ is a free $ {(\mc{P}^{U_I})}^{G(m,1,n_i)} $ module with a basis consisting of $  \beta_{i,n_i,r} $ and $ c_{i,k} (\nmi{T_{i,k}}\Delta_{i,k} \beta_{i,k,r} \alpha_{i,k,j}), $ where $ T\in\mathds{B}(m_{i-1}),$  $ 1\leq k\leq n_i,$  $ 1\leq j\leq C_{n_i}^k,\ 1\leq r\leq a. $
\end{proposition}
\begin{proof}
By above lemmas, 
	$$ f_2=\sum_{k=1}^{n_i}\sum_{ T\in \mathds{B}(m_{i-1})} \sum_{ \substack{S=(s_1,\cdots,s_k)\in\mathds{B}_k \\ m_{i-1}<s_1<\cdots < s_k\leq m_{i}}}\nmi{T+ S}f_{T+ S} $$
	$$=\sum_{k=1}^{n_i}\sum_{  T\in \mathds{B}(m_{i-1})} \sum_{ \substack{S=(s_1,\cdots,s_k)\in\mathds{B}(n) \\  m_{i-1}<s_1<\cdots < s_k\leq m_{i}}}\sigma_{i,S}(\nmi{T_{i,k}}f_{T_{i,k}}) $$
	$$=\sum_{k=1}^{n_i}\sum_{  T\in \mathds{B}(m_{i-1})} c_{i,k} (\nmi{T_{i,k}}f_{T_{i,k}}), $$
	where $ f_{T_{i,k}}\in \mc{P}^{U_I} $ and $ \nmi{T_{i,k}}f_{T_{i,k}} $ is  $ G_{i,k} $ invariant.
	
	Now, for $ g=\on{diag}(w_1,\cdots, w_{n_i})\in A(m,a,n_i), $ one can check that $ g\cdot \nmii{T_{i,k}}=w_1\cdots w_k \nmii{T_{i,k}}.$
	Therefore,
	\begin{equation}\label{aaa}
		\nmi{T_{i,k}}f_{T_{i,k}}=g\cdot\nmi{T_{i,k}}f_{T_{i,k}} =  w_1\cdots w_k\nmi{T_{i,k}}(g\cdot f_{T_{i,k}}). 
	\end{equation}  
	Suppose $ f_{T_{i,k}}=\sum_{\un{j}\in \fn^{n_i}} a_{\un{j}}x^{\un{j}}, $ then $ g\cdot f_{T_{i,k}}= \sum_{\un{j}} a_{\un{j}}w_1^{j_1}\cdots w_{n_i}^{j_{n_i}}x^{\un{j}}.  $ Recall that $ w_i^{m}=(w_1\cdots w_{n_i})^b=1. $ By equation (\ref{aaa}), one have $ a_{\un{j}}=0 $ unless
	\[ j_s= \left\{
	\begin{array}{ll}
	q_sm+rb-1 &, s=1,\cdots,k\\
	q_{s}m+rb &, s=k+1,\cdots,n_i
	\end{array}
	\right. \]
	where $ q_1,\cdots q_{n_i}\in \fn $ and $ r\in \{0,\cdots, a-1\}. $ 
	
	Hence,
	$  f_{T_{i,k}}=\sum_{r=1}^{a} \beta_{i,k,r}\pr{f}_{T,i,k,r}  $ where $ \pr{f}_{T,i,k,r}\in \ff_q[x_{m_{i-1}+1}^m,\cdots,x_{m_{i}}^m]^{U_I}. $
	
	For each $ \sigma\in S_k $ (resp. $\gamma \in S_{n_i-k} $), one can check that $ \sigma (\nmii{T_{i,k}}\beta_{i,k,r})=\det(\sigma)\nmii{T_{i,k}}\beta_{i,k,r} $ (resp. $ \gamma (\nmii{T_{i,k}}\beta_{i,k,r})=\nmii{T_{i,k}}\beta_{i,k,r} $). 
	Since $ \nmii{T_{i,k}}f_{T_{i,k}} $ is $ S_k\times S_{n_i-k} $ invariant,
	we have $ \sigma \pr{f}_{T,i,k,r}=\det(\sigma)^{-1}\pr{f}_{T,i,k,r} $ (resp. $ \gamma \pr{f}_{T,i,k,r}=\pr{f}_{T,i,k,r} $). Namely,
	$ \pr{f}_{T,i,k,r} $ is $ S_k $ skew-invariant and $ S_{n_i-k} $ invariant.
	
	Therefore, there is $ h_{T,i,k,r}\in  \ff_q[x_{m_{i-1}+1}^m,\cdots,x_{m_{i}}^m]^{S_k\times S_{n_i-k}}=\ff_q [x_{m_{i-1}+1},\cdots,x_{m_{i}}]^{H_{i,k}\times \pr{H}_{i,k}} $ such that
	 $  \pr{f}_{T,i,k,r}= \Delta_{i,k}h_{T,i,k,r}.  $ Moreover,
	 $$ {f}_{T_{i,k}}=\sum_{r=1}^{a}\Delta_{i,k} \beta_{i,k,r}\sum_{j=1}^{C_{n_i}^k} \alpha_{i,k,j}f_{T,i,k,r,j}, \text{ where }
	 f_{T,i,k,r,j}\in (\mc{P}^{U_I})^{G(m,1,n_i)}. $$
	
Consequently, 
	\[ f_2=\sum_{k=1}^{n_i}\sum_{  T\in \mathds{B}(m_{i-1})} c_{i,k} (\nmi{T_{i,k}}\sum_{r=1}^{a}\Delta_{i,k} \beta_{i,k,r}\sum_{j=1}^{C_{n_i}^k} \alpha_{i,k,j}f_{T,i,k,r,j}) \]
	\[ =\sum_{k=1}^{n_i}\sum_{  T\in \mathds{B}(m_{i-1})} \sum_{r=1}^{a} \sum_{j=1}^{C_{n_i}^k} c_{i,k} \left(\nmi{T_{i,k}}\Delta_{i,k} \beta_{i,k,r} \alpha_{i,k,j}\right)f_{T,i,k,r,j}. \]
	
Thanks to Proposition \ref{base} and the definition of $ \{\alpha_{i,k,j}, \beta_{i,k,r}\}, $ these generators are linear independent as $ \mc{P}^{G(m,1,n_i)} $ module. 
Proposition holds.	
\end{proof}
\begin{remark}
\begin{enumerate}
	\item  $ \mc{P}^{G(m,a,n_i)} $ is a free $ \mc{P}^{G(m,1,n_i)} $ with a basis $ \{ \beta_{i,n_i,r} \mid r=0,\cdots a-1\}. $ 
	\item Although $ \mc{A}^{G(m,a,n_i)} $ is a $ \mc{P}^{G(m,a,n_i)} $ module, it is hard to formulate the structure as $ \mc{P}^{G(m,a,n_i)} $ module. 
	The main issue is to decompose $ \mc{P}^{G_{i,k}} $ as $ \mc{P}^{G(m,a,n_i)} $ module. 
	\item  $ \mc{P}^{G_{i,k}} $ is complete intersection other than a polynomial ring. In fact, $$ \mc{P}^{G_{i,k}}=\ff_q[u_1,\cdots ,u_{n_i}, v]/ (u_ku_{n_i}-v^a), $$ 
	where 
	\[ u_i=\left\{ 
	\begin{array}{ll}
	\sum_{1\leq j_1< \cdots <j_i\leq k} x_{m_{i-1}+j_1}^m\cdots  x_{m_{i-1}+j_i}^m & i=1,\cdots, k\\
	\sum_{k+1\leq j_1< \cdots <j_{i-k}\leq n_i} x_{m_{i-1}+j_1}^m\cdots  x_{m_{i-1}+j_{i-k}}^m & i=k+1,\cdots, n_i
	\end{array}
	\right.,\] $\text{ and }
	v=(x_1\cdots x_{n_i})^b. $ 
\end{enumerate}
\end{remark}

\begin{corollary}
	If $ a=1, $ i.e. $ G_i=G(m,1,n_i), $ and $ p>n_i, $ then $ \mc{A}^{G(m,1,n_i)} $ is a free $ \mc{P}^{G(m,1,n_i)} $ module with a basis consisting of $ 1 $ and $ c_{i,k} (\nmi{T_{i,k}}\Delta_{i,k}  \alpha_{i,k,j}), $ where $ T\in\mathds{B}(m_{i-1}),$  $ 1\leq k\leq n_i$ and $ 1\leq j\leq C_{n_i}^k. $
\end{corollary}

\subsubsection{$ G_i=\SL_{n_i}(q)\text{ or }\GL_{n_i}(q) $} 
Suppose $ f_2=\sum_{S\leq S^*} \nmi{S}f_S, $ where $ S^*=(s_1^*,\cdots, s_k^*) $ and $ s_j^*< m_{i-1}\leq s_{j+1}^*. $ Let $ U_i $ be the subgroup of $ G_i $ consisting of all upper triangular matrices of the form
\[ \left(
\begin{matrix}
1 & * & \cdots & *\\
0 & 1 & \cdots & *\\
\vdots & \vdots &\vdots &\vdots \\
0 & \cdots & 0 & 1
\end{matrix} 
\right).
\]
\begin{lemma}\label{ui-invarint}
	$ f_{S^*} $ is $ U_i $-invariant.
\end{lemma}
\begin{proof}
	$ \forall u\in U_i,\ u\cdot \nmi{S}=\nmi{S}+\sum_{L<S}a_L\nmi{L}, $ where $ a_L\in \ff_q. $
	Therefore, $$ u\cdot f_2=\nmi{S^*}(u\cdot f_{S^*}) +\sum_{S< S^*} \nmi{S}\pr{f}_S. $$
	$ u\cdot f_2=f_2 $ implies that $ u\cdot f_{S^*}=f_{S^*}. $
\end{proof}

\begin{proposition}\label{5.6}
	\[ f_2= \sum_{ \substack{S=(s_1,\cdots,s_k)\in\mathds{B}_k \\ m_{i-1}< s_k\leq m_{i}}} N_{m_i,S}h_S= \sum_{ \substack{S=(s_1,\cdots,s_k)\in\mathds{B}_k \\ m_{i-1}< s_k\leq m_{i}}} N_{m_i,S}\theta_i^{q-2}\bar{h}_S\]
	where $h_S\in \mc{P}^{\SL_{n_i}}, \bar{h}_S\in \mc{P}^{\GL_{n_i}}. $
\end{proposition}
\begin{proof}
We will use induction on $ S^*. $

For some $ S $ appears in $ f_2, $
denote $ \pr{S}=\{1,\cdots,n\}\xg S. $ For each $ a\in \pr{S}\cap \{m_{i-1}+1,\cdots,m_i\}. $ Suppose $ s_b<a<s_{b+1},$ for some $ 1\leq b\leq k. $ Let 
$$ r=\left\{
\begin{matrix}
s_b & s_b>m_{i-1} \\
s_{b+1} & s_b= m_{i-1}
\end{matrix}\right..  $$ 
Take $ w=E+E_{a,r}\in G. $ Then 
\begin{equation}\label{GL case}
w\cdot f_2=f_2.
\end{equation} 

(i) Suppose $ m_{i-1}<s_b=r. $ 

By comparing the coefficient of $ y_K $ on both side of equation (\ref{GL case}) where $ K=S+\{a\}-\{s_b\}, $ we have \begin{equation}\label{GL subcase}
N_{m_{i-1},K}(w\cdot f_S) + N_{m_{i-1},K}(w\cdot f_K) =N_{m_{i-1},K}f_K.
\end{equation}

In fact, $ w(N_{m_{i-1},J}f_J)=N_{m_{i-1},J}(w\cdot f_J)+ N_{m_{i-1},E_{a,r}\cdot J} (w\cdot f_J) $ has factor $ y_K $ if and only if either $ J=K $ 
or $ E_{a,r}\cdot J=K $ which forces $ J=S.$ 

By Proposition \ref{base}, equation (\ref{GL subcase}) implies that $$ f_S(x_1,\cdots,x_r+x_a,\cdots,x_a,\cdots)= f_K(x_1,\cdots,x_n)-f_K(x_1,\cdots,x_r+x_a,\cdots,x_a,\cdots). $$
Setting $ x_a=0 $ yields
$  f_S(\cdots,x_{a-1},0,x_{a+1},\cdots)=0, $ which implies that $ x_a\mid f_S. $

(ii) Suppose $ m_{i-1}=r, $ i.e. $ s_b\leq m_{i-1}=r<s_{l+1}. $ Similar to (i), by comparing the coefficient of $ y_{\pr{K}} $ on both side of equation (\ref{GL case}) where $ \pr{K}=S+\{a\}- \{s_{l+1}\}, $ we have $  x_a\mid f_S. $

In particular, $ x_a\mid f_{S^*} $ for all $ a\in (S^*)^{\prime}\cap \{m_{i-1}+1,\cdots,m_i\}. $ Thanks to Lemma \ref{ui-invarint}, $ V_a\mid f_{S^*}. $

 By Corollary \ref{action on V's}, we have
\[ f_2=\nmi{S^*}V_{m_{i-1}+1}\cdots \widehat{ V_{S_{j+1}^*}}\cdots \widehat{ V_{S_{k}^*}}\cdots V_{m_i}h_{S^*}+\sum_{S<S^*}\nmi{S}f_S \]
\[ =\nmii{S^*}h_{S^*}+ \sum_{T<(s_1^*,\cdots,s_j^*,m_i-k+j+1,\cdots,m_i)} \nmii{L}h_L+\sum_{S<S^*}\nmi{S}f_S.\]

Since both $ \nmii{S^*} $ and $ \nmii{L} $ are $ \SL_{n_i} $-invariant, then all $ h_{S^*}, h_L $ and $ \sum_{S<S^*}\nmi{S}f_S $ are $ \SL_{n_i} $-invariant. By induction, 
\[f_2= \sum_{ \substack{S=(s_1,\cdots,s_k)\in\mathds{B}_k \\ m_{i-1}< s_k\leq m_{i}}} N_{m_i,S}h_S\]
where $h_S\in \mc{P}^{\SL_{n_i}}. $

Similar to the proof of \cite[Theorem 3.1]{WW}, one have that $ h_S=\theta_i^{q-2} \bar{h}_S $ where $ \bar{h}_S\in \mc{P}^{\GL_{n_i}}. $ 

Proposition holds.
\end{proof}

\begin{corollary}
	\begin{enumerate}
		\item $ (\mc{A}^{U_I})^{\SL_{n_i}} $ is a free $ (\mc{P}^{U_I})^{\SL_{n_i}} $ module with a basis consisting of $$ \{\nmii{S}\mid S\in \mathds{B}(m_i)\xg \mathds{B}(m_{i-1}) \}. $$ 
		\item $ (\mc{A}^{U_I})^{\GL_{n_i}} $ is a free $ (\mc{P}^{U_I})^{\GL_{n_i}} $ module with a basis consisting of $$ \{\nmii{S}\theta_i^{q-2}\mid S\in \mathds{B}(m_i)\xg \mathds{B}(m_{i-1}) \}. $$ 
	\end{enumerate}
\end{corollary}

\section{$ G_I $-invariant of $ \mc{A} $}
In this section, we will apply above results to describe $ \mc{A}^{G_I} $ for some concrete groups $ G_I $ as examples. 

\subsection{} $ G_i=G(r_i,1,n_i) $ for all $ i=1,\cdots, l $ such that $ r_i\mid q-1. $ Suppose $ p>n_i $ for all $ i. $ Hence, all $ G_i $'s are non-modular.

For each $ 1\leq i \leq l,\ T\in \mathds{B}(m_{i-1}),\  1\leq k\leq n_i $ and $ 1\leq j\leq C_{n_i}^k, $  recall the notations $ c_{i,k}, T_{i,k}, \Delta_{i,k}, H_{i,k}, \pr{H}_{i,k} $ and $ \alpha_{i,k,j}  $ in Section \ref{6.2.1}. Furthermore, define
\begin{itemize}
	\item $ u_{i,k}:=e_{i,k}(v_{i,1},\cdots, v_{i,n_i}), $ where $ \ff_q[x_{m_{i-1}+1},\cdots,x_{m_i}]^{G(r_i,1,n_i)}=\ff_q[e_{i,1},\cdots,e_{i,n_i}],$ namely, $  e_{i,j}=\sum_{m_{i-1}+1\leq t_1<\cdots < t_j\leq m_{i}} x_{t_1}^{r_i}\cdots x_{t_j}^{r_i} $ and $ v_{i,k} $ refers to formula (\ref{vij});
	\item 
	 $$ \Omega_{i,k}:=\prod_{t=1}^{i-1}\Delta_{t,n_{t}}\cdot\Delta_{i,k}= \prod_{t=1}^{i-1}\prod_{m_{t-1}<j_1<j_2\leq m_t}(x_{j_1}^{r_t}-x_{j_2}^{r_t}) \cdot \prod_{m_{i-1}<j_1<j_2\leq m_{i-1}+k} (x_{j_1}^{r_i}-x_{j_2}^{r_i}). $$
\end{itemize}

\begin{theorem}\label{imprimitive}
	Suppose $ p>n_i $ for all $ i. $
	\begin{enumerate}
		\item $ \mc{P}^{G_I}=\ff_q[ u_{1,1},\cdots u_{1,n_1},\cdots u_{l,n_l} ]. $
		\item $ \mc{A}^{G_I} $ is a free $ \mc{P}^{G_I} $ module of rank $ 2^{n} $ with a basis consisting of $ 1 $ and $ c_{i,k} (\nmi{T_{i,k}}\Omega_{i,k} \alpha_{i,k,j}), $ where $ 1\leq i\leq l,\  T\in\mathds{B}(m_{i-1}),$  $ 1\leq k\leq n_i$ and $ 1\leq j\leq C_{n_i}^k. $ 
	\end{enumerate}
\end{theorem}
\begin{proof}
(1) By Proposition \ref{poly inv}, statement holds.

(2) For each $ f\in \mc{A}^{G_I}, $  by Proposition \ref{main prop}, suppose 
$$ f=\sum_{S\in \mathds{B}(n)} \n{S}h_S=h_{0}+\sum_{i=1}^lf_i, $$
where \[ h_{0}\in \mc{P}^{U_I},\ f_i=\sum_{\substack{\emptyset\neq S\in \mathds{B}(n)\\\tau(S)=m_{i-1}}} \n{S}h_S, \text{ and } h_S\in \mc{P}^{U_I} \text{ for all } S\in \mathds{B}(n). \]

Now, for each $ 1\leq i\leq l, $ by Lemma \ref{5.11}, \ref{5.2} and Proposition \ref{S action}, $$ f_i
 =\sum_{k=1}^{n_i}\sum_{  T\in \mathds{B}(m_{i-1})} \sum_{j=1}^{C_{n_i}^k} c_{i,k} \left(\nmi{T_{i,k}}\Omega_{i,k} \alpha_{i,k,j}\right)f_{T,i,k,r,j}$$
 \text{ where }
$  f_{T,i,k,r,j}\in (\mc{P}^{U_I})^{G(m,1,n_i)}. $

Consequently, $ \mc{A}^{G_I} $ is generated, as $ \mc{P}^{G_I} $ module, by $ 1 $ and $$\left\{ c_{i,k} \left(\nmi{T_{i,k}}\Omega_{i,k} \alpha_{i,k,j}\right)\mid  1\leq i\leq l, 1\leq k\leq n_i, T\in\mathds{B}(m_{i-1}), 1\leq j\leq C_{n_i}^k\right\}. $$

Thanks to Proposition \ref{base} and the definition of $ \{\alpha_{i,k,j} \}, $ these generators are linear independent as $ \mc{P}^{G_I} $ module. The rank is
\[ 1+\sum_{i=1}^l\sum_{k=1}^{n_i} 2^{m_{i-1}}C_{n_i}^k= 1+\sum_{i=1}^l 2^{m_{i-1}}(2^{n_i}-1)=1+\sum_{i=1}^l (2^{m_{i}}-2^{m_{i-1}})=2^n. \qedhere \]
\end{proof}

\subsection{} $ G_i=G(r_i,a_i,n_i) $ for all $ i $ such that $ r_i=a_ib_i $ and $ r_i\mid q-1. $ Suppose $ p>n_i $ for all $ i. $

For each $ 1\leq i \leq l,\ T\in \mathds{B}(m_{i-1}),\  1\leq k\leq n_i,\ 1\leq r\leq a_i $ and $ 1\leq j\leq C_{n_i}^k, $  recall the definitions $ c_{i,k},\ T_{i,k},\ \Omega_{i,k},\ \beta_{i,k,r} $ and $ \alpha_{i,k,j}  $ in Section \ref{6.2.1}.

Suppose $ \ff_q[x_{m_{i-1}+1},\cdots,x_{m_i}]^{G(r_i,a_i,n_i)} =\ff_q[e_{i,1}, \cdots,e_{i,n_i}],$  define
$ u_{i,k}:=e_{i,k}(v_{i,1},\cdots, v_{i,n_i}), $ where $ v_{i,k} $ refers to formula (\ref{vij}).

Denote $ \overline{G_I}:=(G(r_1,1,n_1)\times \cdots\times G(r_l,1,n_l))\ltimes U_I. $ 
For convenience, denote $ \beta_{\un{s}}=\beta_{1,n_1,s_1}\cdots \beta_{l,n_l,s_l} $ where $ \un{s}=(s_1,\cdots s_l)$ such that $ 1\leq s_i \leq a_i $ for all $ i. $ By Lemma \ref{6.4}, we have
\begin{lemma}
	$ \mc{P}^{G_I} $ is a free $ \mc{P}^{\overline{G_I}} $ module of rank $ (a_1\cdots a_l) $ with a basis consisting of $ \beta_{\un{s}} $ for all $ \un{s}.$
\end{lemma}

Similar to above arguments, we can prove the following result.
\begin{theorem}\label{imprimi 2}
	Suppose $ p>n_i $ for all $ i. $
	\begin{enumerate}
		\item $ \mc{P}^{G_I}=\ff_q[ u_{1,1},\cdots u_{1,n_1},\cdots u_{l,n_l} ]. $
		\item $ \mc{A}^{G_I} $ is a free $ \mc{P}^{\overline{G_I}} $ module of rank $ (2^{n}a_1\cdots a_l) $ with a basis consisting of $ \beta_{\un{s}} $ and \\$ c_{i,k} \left(\nmi{T_{i,k}}\Omega_{i,k} \alpha_{i,k,j}\beta_{\un{s}}\right), $ where $ 1\leq i,\pr{i}\leq l,\  T\in\mathds{B}(m_{i-1}),$  $ 1\leq k\leq n_i, 1\leq j\leq C_{n_i}^k,$  $ 1\leq r\leq a_i $ and $ \un{s}=(s_1,\cdots,s_l). $
	\end{enumerate}
\end{theorem}

\subsection{} Suppose there is $ 1\leq a\leq l $ such that 

\[ G_i=\left\{
\begin{array}{ll}
\GL_{n_i}(q) & i=1,\cdots,a \\
G(r_i,1,n_i) & i=a+1,\cdots ,l
\end{array}
\right. \]
and $ p>n_i $ for $ i=a+1,\cdots, l. $

For each $ 1\leq i \leq l,\  1\leq k\leq n_i $ and $ 1\leq j\leq C_{n_i}^k, $ recall that $ q_{i,k} $ is defined as formula (\ref{qij}), $ c_{i,k}, u_{i,k}, T_{i,k}, H_{i,k},$  $ \pr{H}_{i,k} $ and $ \alpha_{i,k,j} $ is defined in Section \ref{6.2.1}.

Moreover, if $ a<i\leq l, $ by \cite[section 20-2]{Ka}, all skew-invariants of $ G_{a+1}\times\cdots \times G_{i-1}\times H_{i,k} $ form a free $ \ff_q[x_{1}, \cdots,x_{m_{i-1}+k}]^{G_{a+1}\times\cdots \times G_{i-1}\times H_{i,k}} $ module with one generator, which is denoted by $ \Omega_{i,k}^{(a)}. $ In fact,	
$$ \Omega^{(a)}_{i,k}:=\prod_{t=a+1}^{i-1}\prod_{m_{t-1}<j_1<j_2\leq m_t}(x_{j_1}^{r_t}-x_{j_2}^{r_t}) \cdot \prod_{m_{i-1}<j_1<j_2\leq m_{i-1}+k} (x_{j_1}^{r_i}-x_{j_2}^{r_i}). $$

\begin{theorem}\label{6.3}
	Suppose $ p>n_i $ for $ i=a+1,\cdots, l. $
	\begin{enumerate}
		\item $ \mc{P}^{G_I}=\ff_q[ q_{1,1},\cdots q_{1,n_1},\cdots q_{a,n_a}, u_{a+1,1},\cdots u_{a+1,n_{a+1}},\cdots u_{l,n_l} ]. $
		\item $ \mc{A}^{G_I} $ is a free $ \mc{P}^{G_I} $ module of rank $ 2^{n} $ with a basis consisting of $ 1, $ 
		\[ \left\{ N_{m_i,S}\theta_1^{q-2}\cdots \theta_i^{q-2}\mid 1\leq i\leq a, S=(s_1,\cdots,s_k)\in \mathds{B}(m_i), s_1>m_{i-1} \right\} 
		\text{ and } \]
		$$ \left\{c_{i,k} \left(\nmi{T_{i,k}}\Omega^{(a)}_{i,k} \alpha_{i,k,j}\right)\theta_1^{q-2}\cdots \theta_a^{q-2} \mid  a+1\leq i\leq l,\ 1\leq k\leq n_i,\ T\in\mathds{B}(m_{i-1}),\ 1\leq j\leq C_{n_i}^k\right\}. $$
	\end{enumerate}
\end{theorem}
\begin{proof}
	(1) By Proposition \ref{poly inv}, statement holds.
	
	(2) Suppose $ f=f_0+f_1+f_2\in \mc{A}^{G_I}, $ where $ f_0\in \mc{P}^{G_I}, $
	\[ f_1=\sum_{\emptyset\neq S\in \mathds{B}(m_a)} \n{S}h_S \text{ and }f_2= \sum_{ \substack{ S\in \mathds{B}(n)\\\tau(S)\geq m_{a-1}}} \n{S}h_S. \]
	
	Note that $ f_1 $ is $ G_1\times\cdots\times  G_l $ invariant. By Proposition \ref{5.6}, $$ f_1=\sum_{i=1}^a\sum_{ \emptyset\neq \substack{ S\in \mathds{B}(m_a)\\\tau(S)= m_{i-1}}} N_{m_i},S\theta_1^{q-2}\cdots \theta_i^{q-2}\pr{h}_S $$ where $ \pr{h}_S\in \mc{P}^{G_I}. $
	
By Lemma \ref{5.2} and Proposition \ref{S action},
	\[ f_2=\sum_{i=a+1}^l\sum_{k=1}^{n_i}\sum_{  T\in \mathds{B}(m_{i-1})} \sum_{j=1}^{C_{n_i}^k} c_{i,k} \left(\nmi{T_{i,k}}\Omega^{(a)}_{i,k}\cdot  \alpha_{i,k,j}\right)h_{T,i,k,j}  \]
	where $ h_{T,i,k,j}\in \mc{P}^{U_I} $ is $ G_1\times\cdots \times G_a $ skew-invariant and $ G_{a+1}\times\cdots \times G_l $ invariant. 
	
	Similar to the proof of \cite[Theorem 3.1]{WW}, $ h_{T,i,k,j} =\theta_1^{q-2}\cdots \theta_a^{q-2}\pr{h}_{T,i,k,j}, $ where $ \pr{h}_{T,i,k,j} \in \mc{P}^{G_I}. $ 
	
	Theorem holds.	
\end{proof}

\subsection{Weyl groups of Cartan type Lie algebras}
As a corollary, suppose $ G_I $ is a Weyl group of Cartan type Lie algebra $ \mf{g} $ of type $ W, S $ or $ H. $ Precisely, by \cite{Je}, 
\begin{equation}\label{6.111}
	G_I=\left\{
	\left(\begin{matrix}
		A & B\\ 0 &C
	\end{matrix}\right)
	\bigg|\ A\in \GL_{n_1}(p),\ C\in G_2 \right\}<\GL_n(p),
\end{equation}
where $ G_2=
\left\{
\begin{array}{ll}
S_{n_2} &\text{ if }\lieg\text{ is of type } W \text{ or } S, \\
G(2,1,n_2) &\text{ if } \mf{g} \text{ is of type }  H.  
\end{array}
\right.
$

Recall that $ \mc{P}^{U_I}=\ff_q[x_1,\cdots,x_{n_1},v_{1,1},\cdots v_{1,n_2}] $ and
$ \ff_q[x_{n_1+1},\cdots,x_{n}]^{S_{n_2}}=\ff_q[e_1,\cdots,e_{n_2}] $ where $ e_j=\sum_{n_1+1\leq i_1<\cdots<i_j\leq n} x_{i_1}\cdots x_{i_j}. $ Define $ u_i=e_i(v_{1,1},\cdots v_{1,n_2}). $ 

The following is a direct corollary of Theorem \ref{6.3}.
\begin{corollary}\label{Weyl gp inv}
	Suppose $ p>n_2. $
	\begin{enumerate}
		\item $\mc{P}^{G_I}=\ff_q[Q_{n_1,0},\cdots, Q_{n_1,n_1-1}, u_{2,1},\cdots, u_{2,n_2}].  $
		\item $ \mc{A}^{G_I} $ is a free $ \mc{P}^{G^I} $ module of rank $ 2^n $ with a basis consisting of $ 1, $ 
		\[ \{ N_{n_1,S}L_{n_1}^{q-2}\mid  S\in \mathds{B}(n_1)\xg\mathds{B}_0 \} 
		\text{ and } \]
		$$ \left\{c_{k} \left(N_{n_1,T_{1,k}}\Omega^{(1)}_{1,k} \alpha_{1,k,j}\right)L_{n_1}^{q-2} \mid   1\leq k\leq n_2,\ T\in\mathds{B}(n_1),\ 1\leq j\leq C_{n_i}^k\right\}. $$
	\end{enumerate}
\end{corollary}

\bigskip
\noindent
\textbf{Acknowledgments.}
This work is supported by NSFC (No. 12101544). We are indebted to the referee for carefully reading the manuscript and providing numerous comments.

\end{document}